\documentclass[12pt]{article}

\usepackage{amsmath}
\usepackage{amssymb}
\usepackage{amsthm}
\usepackage{cases}
\usepackage[dvips]{graphicx}
\usepackage{bm}

\newtheorem{theorem}{Theorem}
\newtheorem{proposition}[theorem]{Proposition}

\newtheorem{lemma}[theorem]{Lemma}
\theoremstyle{definition}
\newtheorem{definition}[theorem]{Definition}

\newtheorem{remark}[theorem]{Remark}
\newtheorem{example}[theorem]{Example}

\newcommand{\N}{\mathbb{N}}

\newcommand{\Z}{\mathbb{Z}}

\newcommand{\SL}{{\mbox{\rm SL}}}

\begin{document}

\title{Homology for Quandles with Partial Group Operations}
\author{Scott Carter${}^1$, Atsushi Ishii${}^2$, Masahico Saito${}^3$ and Kokoro Tanaka${}^4$ \\[3mm]
${}^1$University of South Alabama \ 
${}^2$University of Tsukuba \\
${}^3$University of South Florida \  ${}^4$Tokyo Gakugei University
}

\date{}

\maketitle

\begin{abstract}
A quandle is a set that has a binary operation satisfying three conditions
corresponding to the Reidemeister moves.
Homology theories of quandles have been developed in a way similar to group homology, and have been applied to knots and knotted surfaces. 
In this paper, a homology theory is defined that unifies group and quandle homology theories.
A quandle that is a union of groups with the operation restricting to conjugation on each group component 
is called a multiple conjugation quandle (MCQ, 
defined rigorously within). 
In this definition, 
 compatibilities between the group and quandle operations are imposed which are motivated by considerations on colorings of handlebody-links. 
A homology theory defined here for MCQs take into consideration both group and quandle operations,
as well as their compatibility. 
 The first homology group is characterized, and the notion of extensions by $2$-cocycles is provided. Degenerate subcomplexes are defined 
in relation to simplicial decompositions of prismatic (products of simplices) complexes and group inverses. 
Cocycle invariants are 
also 
defined for handlebody-links.  
\end{abstract}

\section{Introduction}

In this paper, a homology theory is proposed that contains aspects of both group and quandle 
homology theories, for algebraic structures 
that have both operations  and  certain 
compatibility conditions between them.

The notion of a quandle \cite{Joyce82,Matveev82} was introduced in knot theory
as a generalization of  the fundamental group.
Briefly, a quandle is a set with a binary operation that is idempotent, self-distributive, and
the corresponding right action is bijective. 
The  axioms  correspond  to  the Reidemeister moves, and quandles have been  used extensively 
to construct  knot invariants.
They have been considered in various other contexts, for example as symmetries of 
geometric objects \cite{Takasaki42}, and 
 with different names,  such as  distributive groupoids~\cite{Matveev82} and 
 automorphic sets~\cite{Bries}. 
A typical example is a group conjugation $a*b=b^{-1}ab$ which is an expression of the Wirtinger relation for the 
fundamental group of the knot complement.
The same structure but without idempotency is called a rack, and  
is used  in the study of  framed links~\cite{FR}.

In~\cite{FRS95} a chain complex was introduced for racks.
The resulting homology theory was modified in~\cite{CJKLS} by defining a quotient complex that reflected the quandle 
idempotence  axiom.  
The motivation for this homology was to construct the quandle cocycle invariants for links and surface-links.
Since then a variety of applications have been found.
The quandle cocycle invariants were generalized to handlebody-links in
\cite{IshiiIwakiri12}.
When a set has multiple 
quandle operations that are parametrized by a group, the structure is called a $G$-family of quandles; this notion  and its associated homology theory was introduced in \cite{IshiiIwakiriJangOshiro13} and it too was motivated from handlebody-knots.  
This homology theory is called {\it IIJO}. 
In particular, cocycle invariants were introduced that distinguished 
mirror images of 
some of  handlebody-knots. 
These $G$-families were further generalized to 
an algebraic system  called a multiple conjugation quandle (MCQ) in \cite{Ishii_MCQ} for colorings of handlebody-knots.
An MCQ has a quandle operation and partial group operations, with compatibility conditions among the operations.

The current paper proposes to unify the group and quandle homology theories for MCQs. 
The definition of an MCQ 
 is recalled in Section~\ref{sec:MCQ} as a generalization of a $G$-family of quandles. 
A homology theory is defined (in Section~\ref{sec:homology}) that simultaneously encompasses the group 
and quandle homologies of the interrelated structures. 
As in the case of \cite{CJKLS}, some subcomplexes are defined in order to compensate 
for the topological motivation of the theory. 
The first homology group is characterized, 
and the notion of extensions by $2$-cocycles is provided in Section~\ref{sec:alg}.

The homology theory for MCQ is well suited for handlebody-links such that 
each toroidal component has its core circle oriented, as defined  
in Section~\ref{sec:coloring}.  
When considering colorings for unoriented handlebody-links, 
we also need to take into consideration issues 
about the inverse elements in the group (Section~\ref{sec:unori}). 
Prismatic sets (products of simplices) 
are decomposed into sub-simplices, 
that are higher dimensional analogues of graph moves;
Section~\ref{sec:decomp} defines a subcomplex that compensates for these subdivisions. 
In Section~\ref{sec:G} and Section~\ref{sec:IIJO}, 
we relate this homology theory with group and quandle homology theories. 
Finally, in Section~\ref{sec:attempt}, we discuss 
approaches to finding new $2$-cocycles of our homology theory.

\section{Multiple conjugation quandles}
\label{sec:MCQ}

First, recall a \textit{quandle}~\cite{Joyce82,Matveev82}, is a non-empty set $X$ with a binary operation $*:X\times X\to X$ satisfying the following axioms.
\begin{itemize}
\item[(1)] For any $a\in X$, we have $a*a=a$.
\item[(2)] For any $a\in X$, the map $S_a:X\to X$ defined by $S_a(x)=x*a$ is a bijection.
\item[(3)] For any $a,b,c\in X$, we have $(a*b)*c=(a*c)*(b*c)$.
\end{itemize}

\begin{definition}[\cite{Ishii_MCQ}]
A \textit{multiple conjugation quandle (MCQ)} $X$ is the disjoint union of groups $G_\lambda$,
where $\lambda$ is an element of an index set $\Lambda$, 
with a binary operation $*:X\times X\to X$ satisfying the following axioms.
\begin{itemize}
\item[(1)] For any $a,b\in G_\lambda$, we have $a*b=b^{-1}ab$.
\item[(2)] For any $x\in X$, $a,b\in G_\lambda$, we have $x*e_\lambda=x$ and $x*(ab)=(x*a)*b$, where $e_\lambda$ is the identity element of $G_\lambda$.
\item[(3)] For any $x,y,z\in X$, we have $(x*y)*z=(x*z)*(y*z)$.
\item[(4)] For any $x\in X$, $a,b\in G_\lambda$, 
we have $(ab)*x=(a*x)(b*x)$ in some group $G_\mu$.
\end{itemize}
\end{definition}

We call the group $G_\lambda$ a \textit{component} of the MCQ.
An MCQ is a type of quandles that can be decomposed as a union of groups, and the quandle operation in each component is given by conjugation.
Moreover, there are compatibilities, (2) and (4),  between the group and quandle operations.

Note that the quandle axiom $a*a=a$ follows immediately since the operation in any component is given by conjugation.
The second quandle axiom also follows, since for the map $S_a: X \rightarrow X$ defined by $S_a(x)=x*a$,
the inverse map is given by $S_{a^{-1}}$.
The second axiom of MCQ  
implies that the map $\varphi:G_\lambda\to\operatorname{Aut}_\mathsf{Qnd}X$ defined by $\varphi(a)=S_a$ is a group homomorphism, where $\operatorname{Aut}_\mathsf{Qnd}X$ is the set of quandle automorphisms of $X$ and is the group with the multiplication defined by $S_aS_b:=S_b\circ S_a$. 
The last axiom (4) may be replaced with
\begin{itemize}
\item[($4'$)]
For any $x\in X$ and $\lambda\in\Lambda$, there is a unique element $\mu\in\Lambda$ such that $S_x(G_\lambda)=G_\mu$ and that $S_x:G_\lambda\to G_\mu$ is a group isomorphism.
\end{itemize}
The axiom (4) immediately follows from $(4')$. 
Conversely, $(4')$  
follows from $(4)$: 
The condition (4) contains the condition that for any $a, b \in G_\lambda$ and $x \in X$, there exists a unique $\mu \in \Lambda$ such that $a*x, b*x \in G_\mu$. Hence we have 
 $S_x(G_\lambda)\subset G_\mu$, which implies that $S_x:G_\lambda\to G_\mu$ is a well-defined group homomorphism by the condition $(ab)*x=(a*x)(b*x)$.
The homomorphism $S_x:G_\lambda\to G_\mu$ is a group isomorphism, since $S_{x^{-1}}:G_\mu\to G_\lambda$ gives its inverse.

A multiple conjugation quandle can be obtained from a $G$-family of quandles as follows.

\begin{example} \label{ex:mcq from G-family}
Let $G$ be a group with identity element $e$, let $(M,\{*^g\}_{g\in G})$ be a $G$-family of quandles~\cite{IshiiIwakiriJangOshiro13}; \textit{i.e.} a non-empty set $M$ with a family of binary operations $*^g:M\times M\to M$ ($g\in G$) satisfying
\begin{align*}
&x*^gx=x,~~~x*^{gh}y=(x*^gy)*^hy,~~~x*^ey=x, \\
&(x*^gy)*^hz=(x*^hz)*^{h^{-1}gh}(y*^hz)
\end{align*}
for $x,y,z\in M$, $g,h\in G$.
Then $\coprod_{x\in M}\{x\}\times G$ is a multiple conjugation quandle with
\begin{align*}
&(x,g)*(y,h)=(x*^hy,h^{-1}gh),
&&(x,g)(x,h)=(x,gh).
\end{align*}

The following are specific examples of $G$-families of quandles.
\begin{itemize}
\item[(1)]
Let $M$ be a group, and $G$ be a subgroup of $\operatorname{Aut}M$.
Then for $x,y\in M$ and $g\in G$, $x*y=(xy^{-1})^gy$ gives a $G$-family of quandles.
Here $x^g$ denotes $g$ acting on $x$.
The fact that this is a $G$-family was pointed out by Przytycki (cf.~\cite{Przytycki11}); however, that any specific automorphism $g$ yields a quandle was 
earlier observed in  \cite{Joyce82,Matveev82}.  
When $M$ is abelian and an element $g\in G$ is fixed, the resulting quandle is called an Alexander quandle.
\item[(2)]
Let $(X,*)$ be a quandle.
We denote $S_b^n(a)$ by $a*^nb$.
Put $Z:=\mathbb{Z}$ or $\mathbb{Z}/m\mathbb{Z}$, where
$m:=\min\{i>0\,|\,\text{$x*^iy=x$ for any $x,y\in X$}\}$.
Then $(X,\{*^n\}_{n\in Z})$ is a $Z$-family of quandles.
\end{itemize}
\end{example}

For a multiple conjugation quandle $X=\coprod_{\lambda\in\Lambda}G_\lambda$, an \textit{$X$-set} is a non-empty set $Y$ with a map $*:Y\times X\to Y$ satisfying the following axioms, where 
we use the same symbol $*$ as the binary operation of $X$.
\begin{itemize}
\item For any $y\in Y$ and $a,b\in G_\lambda$, we have $y*e_\lambda=y$ and $y*(ab)=(y*a)*b$, where $e_\lambda$ is the identity of $G_\lambda$.
\item For any $y\in Y$ and $a,b\in X$, we have $(y*a)*b=(y*b)*(a*b)$.
\end{itemize}
Any multiple conjugation quandle $X$ itself is an $X$-set with its binary operation.
Any singleton set $\{y_0\}$ is also an $X$-set with the map $*$ defined by $y_0*x=y_0$ for $x\in X$, which is 
called a trivial $X$-set.
The index set $\Lambda$ is an $X$-set with the map $*$ defined by 
$ \lambda *x = \mu$ when $S_x(G_\lambda)=G_\mu $ for $\lambda, \mu \in\Lambda$ and $x\in X$.

\section{Homology theory}
\label{sec:homology}

In this section, we define a chain complex for MCQs that contains aspects of both group and quandle 
homology theories. A subcomplex is also defined that corresponds to a Reidemeister move for handlebody-links.

Let $X=\coprod_{\lambda\in\Lambda}G_\lambda$ be a multiple conjugation quandle, and let $Y$ be an $X$-set.
In what follows, we denote a sequence of elements of $X$ by a bold symbol such as $\boldsymbol{a}$, and denote by $|\boldsymbol{a}|$ the length of a sequence $\boldsymbol{a}$.
For example, $(\boldsymbol{a})$, $\langle\boldsymbol{a}\rangle$,
$(y;\boldsymbol{a};\boldsymbol{b})$ respectively denote
\begin{align*}
&(a_1,\ldots,a_{|\boldsymbol{a}|}),
&&\langle a_1,\ldots,a_{|\boldsymbol{a}|}\rangle,
&&(y;a_1,\ldots,a_{|\boldsymbol{a}|};b_1,\ldots,b_{|\boldsymbol{b}|}).
\end{align*}
Let $P_n(X)_Y$ be the free abelian group generated by the elements
\[ (y;a_{1,1},\ldots,a_{1,n_1};\ldots;a_{k,1},\ldots,a_{k,n_k})\in
\bigcup_{n_1+\cdots+n_k=n}Y\times\prod_{i=1}^k
\bigcup_{\lambda\in\Lambda}G_\lambda^{n_i} \]
if $n\geq0$, and let $P_n(X)_Y=0$ otherwise.
The generators of $P_n(X)_Y$ are called \textit{prismatic chains} and $P_n(X)_Y$ is called the \textit{prismatic chain group}.
Note that for each $j$, the elements $a_{j,1},\ldots,a_{j,n_j}$ belong to one of $G_\lambda$'s.
For example, $P_3(X)_Y$ is generated by the elements $(y;a;b;c)$, $(y;a;e,f)$, $(y;d,e;c)$ and $(y;d,e,f)$ ($a,b,c\in X$, $d,e,f\in G_\lambda$, $y\in Y$).
Here $a,b,c$ may or may not belong to the same $G_\mu$ ($\mu \in \Lambda$), but $d,e,f$ belong to the same $G_\lambda$.
All may belong to the same $G_\lambda$.

We use the noncommutative multiplication form
\[ \langle y\rangle
\langle\boldsymbol{a}_1\rangle\cdots\langle\boldsymbol{a}_k\rangle \]
to represent $(y;\boldsymbol{a}_1;\ldots;\boldsymbol{a}_k)$.
We define
\[ \langle y\rangle
\langle\boldsymbol{a}_1\rangle\cdots\langle\boldsymbol{a}_k\rangle*b
:=\langle y*b\rangle\langle
\boldsymbol{a}_1*b\rangle\cdots\langle\boldsymbol{a}_k*b\rangle, \]
where $\langle\boldsymbol{a}*b\rangle$ denotes $\langle a_1*b,\ldots,a_{|\boldsymbol{a}|}*b\rangle$.
We set $|\langle y\rangle
\langle\boldsymbol{a}_1\rangle\cdots\langle\boldsymbol{a}_k\rangle|
:=|\boldsymbol{a}_1|+\cdots+|\boldsymbol{a}_k|$.

We define a boundary homomorphism $\partial_n:P_n(X)_Y\to P_{n-1}(X)_Y$ by
\[ \partial(\langle y\rangle
\langle\boldsymbol{a}_1\rangle\cdots\langle\boldsymbol{a}_k\rangle)
=\sum_{i=1}^k(-1)^{|\langle y\rangle
\langle\boldsymbol{a}_1\rangle\cdots\langle\boldsymbol{a}_{i-1}\rangle|}
\langle y\rangle
\langle\boldsymbol{a}_1\rangle\cdots\partial\langle\boldsymbol{a}_i\rangle
\cdots\langle\boldsymbol{a}_k\rangle, \]
where
\begin{align*}
\partial\langle a_1,\ldots,a_m\rangle
&=*a_1\langle a_2,\ldots,a_m\rangle
+\sum_{i=1}^{m-1}(-1)^i\langle a_1,\ldots,a_ia_{i+1},\ldots,a_m\rangle \\
&\hspace{5mm}+(-1)^m\langle a_1,\ldots,a_{m-1}\rangle.
\end{align*}
The resulting terms
$\partial(\langle a\rangle)=*a\langle~\rangle-\langle~\rangle$
for $m=1$ in the above expression means that the formal symbol $\langle~\rangle$ is deleted.
For $n=0$, we define $\partial\langle y\rangle=0$.

\begin{example}
The boundary maps in 2- and 3-dimensions are computed as follows. 
\begin{align*}
\partial_2(\langle y\rangle\langle a\rangle\langle b\rangle)
&=\langle y*a\rangle\langle b\rangle-\langle y\rangle\langle b\rangle
-\langle y*b\rangle\langle a*b\rangle+\langle y\rangle\langle a\rangle, \\
\partial_2(\langle y\rangle\langle a,b\rangle)
&=\langle y*a\rangle\langle b\rangle-\langle y\rangle\langle ab\rangle
+\langle y\rangle\langle a\rangle, \\
\partial_3(\langle y\rangle\langle a\rangle\langle b\rangle\langle c\rangle)
&=\langle y*a\rangle\langle b\rangle\langle c\rangle
-\langle y\rangle\langle b\rangle\langle c\rangle
-\langle y*b\rangle\langle a*b\rangle\langle c\rangle \\
&\hspace{5mm}+\langle y\rangle\langle a\rangle\langle c\rangle
+\langle y*c\rangle\langle a*c\rangle\langle b*c\rangle
-\langle y\rangle\langle a\rangle\langle b\rangle, \\
\partial_3(\langle y\rangle\langle a\rangle\langle b,c\rangle)
&=\langle y*a\rangle\langle b,c\rangle
-\langle y\rangle\langle b,c\rangle
-\langle y*b\rangle\langle a*b\rangle\langle c\rangle \\
&\hspace{5mm}+\langle y\rangle\langle a\rangle\langle bc\rangle
-\langle y\rangle\langle a\rangle\langle b\rangle, \\
\partial_3(\langle y\rangle\langle a,b\rangle\langle c\rangle)
&=\langle y*a\rangle\langle b\rangle\langle c\rangle
-\langle y\rangle\langle ab\rangle\langle c\rangle
+\langle y\rangle\langle a\rangle\langle c\rangle \\
&\hspace{5mm}+\langle y*c\rangle\langle a*c,b*c\rangle
-\langle y\rangle\langle a,b\rangle, \\
\partial_3(\langle y\rangle\langle a,b,c\rangle)
&=\langle y*a\rangle\langle b,c\rangle
-\langle y\rangle\langle ab,c\rangle
+\langle y\rangle\langle a,bc\rangle
-\langle y\rangle\langle a,b\rangle.
\end{align*}
\end{example}

\begin{proposition}
$P_*(X)_Y=(P_n(X)_Y,\partial_n)$ is a chain complex.
\end{proposition}
 
\begin{proof}
The Leibniz rule 
\[ \partial(\sigma\tau)
=(\partial\sigma)\tau+(-1)^{| \sigma |}\sigma(\partial\tau) \]
is a restatement of the definition when $k=2$. In fact, the general definition follows from this by induction. 
Also $\partial(\sigma*a)=(\partial\sigma)*a$, and $\partial\circ\partial=0$ follows from these two facts.
\end{proof}

In the sequel, we will define a 
 degeneracy subcomplex 
 that is analogous (albeit more complicated) to the sub-complex of degeneracies for quandle homology. 
Before its definition, we give a description of simplicial decompositions of products of simplices for motivation. 
We identify an $n$-simplex $\Delta^n$ with the set 
 $\{ (x_1,x_2,\ldots, x_n)\in[0,1]^n: 0\le x_1 \le x_2 \le \cdots \le x_n \le 1\} $ 
 called the {\it right $n$-simplex}. 
 Then the $n$-cube 
$[0,1]^n$ can be decomposed into $n!$ sets each of which is congruent to this  right $n$-simplex that has $n$ edges of length $1$, and 
has $(n-k+1)$ edges of length $\sqrt{k}$ 
 for $k=1, \ldots, n$. 
 More specifically,
 for 
 $\vec{x} \in [0,1]^n$ consider the permutation $\sigma\in \Sigma_n$ such that $0\le x_{\sigma(1)} \le x_{\sigma(2)} \le \cdots \le x_{\sigma(n)} \le 1 $. If the coordinates of $\vec{x}$ are 
 all 
 distinct, then there is a unique such $\sigma$
 and an $n$-simplex $\Delta^n_\sigma$ congruent to the right $n$-simplex
 such that   $\vec{x}$ lies
 in the interior of $\Delta^n_\sigma$. 
 Otherwise 
 $\vec{x}$ lies in the boundary of more than one such simplex. 
 Now consider 
 the product of right simplices 
 $$\Delta^s \times \Delta^t =\{(\vec{x},\vec{y})\in [0,1]^{s+t}: 
0\le x_1 \le x_2 \le \cdots \le x_s \le 1 \ \& \ 0\le y_1 \le y_2 \le \cdots \le y_t \le 1 \} $$ 
where the notation  $(\vec{x}, \vec{y})$ represents $(x_1, \ldots, x_s, y_1, \ldots, y_t)$.
 This   can be decomposed as a union of 
   simplices    of the form given 
   above. 
      For $\vec{z}=(\vec{x},\vec{y})\in \Delta^s \times \Delta^t \subset [0,1]^n$ where $n=s+t$, there is an associated simplex $\Delta^n_\sigma$ that contains the point $(\vec{x},\vec{y})$. 
Suppose all coordinates of $\vec{z}$ are distinct, and let $\sigma \in \Sigma_n$ be a permutation such that 
$0 < z_{\sigma(1)} < \cdots < z_{\sigma(n)}$. 
Then the subset $\{ i_1, i_2, \ldots , i_s \} \subset \{1,2, \ldots, s+t \}$ with $i_1 < i_2 < \cdots < i_s$  
is determined from the positions of coordinates of $\vec{x}$, 
 so that $z_{i_k}=x_k$ for $k=1, \ldots, s$. 
Thus a given  subset $\{ i_1, i_2, \ldots , i_s \} \subset \{1,2, \ldots, s+t \}$ where $i_1 < i_2 < \cdots < i_s$  determines an $n$-simplex in the decomposition of  $\Delta^s \times \Delta^t$.
We proceed to the definition of the degeneracy subcomplex.

For an expression of the form
$\langle\boldsymbol{a}\rangle\langle\boldsymbol{b}\rangle$ in a chain in $P_n(X)_Y$, where
$\langle\boldsymbol{a}\rangle=\langle a_1,\ldots,a_s\rangle$ and
$\langle\boldsymbol{b}\rangle=\langle b_1,\ldots,b_t\rangle$ satisfy
$a_i,b_j\in G_{\lambda}$ for all $i=1,\ldots,s$ and $j=1,\ldots,t$, we define the notation
$\langle\langle\boldsymbol{a}\rangle
\langle\boldsymbol{b}\rangle\rangle_{i_1,\ldots,i_s}$ 
to represent
$(-1)^{\sum_{k=1}^s(i_k-k)}\langle c_1,\ldots,c_{s+t}\rangle$,
where $1\leq i_1<\cdots<i_k<\cdots<i_s\leq s+t$, and
\[ c_i=\begin{cases}
a_k*(b_1\cdots b_{i-k}) & \text{if $i=i_k$,} \\
b_{i-k} & \text{if $i_k<i<i_{k+1}$.}
\end{cases} \]
If $i=k$ in the first case, then we regard $(b_1\cdots b_{i-k})$ to be empty.
For example,
$\langle\langle a\rangle\langle b\rangle\rangle_{1}=\langle a,b\rangle$,
$\langle\langle a\rangle\langle b\rangle\rangle_{2}=-\langle b,a*b\rangle$,
and $\langle\langle a,b\rangle\langle c\rangle\rangle_{1,3}=-\langle a,c,b*c\rangle$.
We also define the notation
$\langle\langle\boldsymbol{a}\rangle\langle\boldsymbol{b}\rangle\rangle$ by
\[ \langle\langle\boldsymbol{a}\rangle\langle\boldsymbol{b}\rangle\rangle
:=\sum_{1\leq i_1<\cdots<i_s\leq s+t}
\langle\langle\boldsymbol{a}\rangle\langle\boldsymbol{b}\rangle\rangle_{i_1,\ldots,i_s}. \]

Define $D_n(X)_Y$ to be the subgroup of $P_n(X)_Y$ generated by the elements of the form
\[ \langle y\rangle\langle\boldsymbol{a}_1\rangle\cdots
\langle\boldsymbol{a}\rangle\langle\boldsymbol{b}\rangle
\cdots\langle\boldsymbol{a}_k\rangle
-\langle y\rangle\langle\boldsymbol{a}_1\rangle\cdots
\langle\langle\boldsymbol{a}\rangle\langle\boldsymbol{b}\rangle\rangle
\cdots\langle\boldsymbol{a}_k\rangle \]
where we implicitly assume the linearity of the notations
$\langle\langle\boldsymbol{a}\rangle
\langle\boldsymbol{b}\rangle\rangle_{i_1,\ldots,i_s}$
and $\langle\langle\boldsymbol{a}\rangle\langle\boldsymbol{b}\rangle\rangle$, that is,
\[ \langle y\rangle\langle\boldsymbol{a}_1\rangle\cdots
\langle\langle\boldsymbol{a}\rangle\langle\boldsymbol{b}\rangle\rangle
\cdots\langle\boldsymbol{a}_k\rangle
=\sum_{1\leq i_1<\cdots<i_{|\boldsymbol{a}|}\leq|\langle\boldsymbol{a}\rangle\langle\boldsymbol{b}\rangle|}
\langle y\rangle\langle\boldsymbol{a}_1\rangle\cdots
\langle\langle\boldsymbol{a}\rangle\langle\boldsymbol{b}\rangle
\rangle_{i_1,\ldots,i_{|\boldsymbol{a}|}}\cdots\langle\boldsymbol{a}_k\rangle. \]

The chain group $D_n(X)_Y$ is called the group of \textit{decomposition degeneracies}.
We will see that $D_*(X)_Y=(D_n(X)_Y,\partial_n)$ is a subcomplex of $P_*(X)_Y$ in Section~\ref{sec:decomp}.

We remark that the elements of the form
\[ \langle y\rangle\langle\boldsymbol{a}_1\rangle\cdots\langle a\rangle
\langle a\rangle\cdots\langle\boldsymbol{a}_k\rangle \]
belong to $D_n(X)_Y$.

For example, $D_2(X)_Y$ is generated by the elements of the form
\[ \langle y\rangle\langle a\rangle\langle b\rangle
-\langle y\rangle\langle a,b\rangle
+\langle y\rangle\langle b,a*b\rangle, \]
and $D_3(X)_Y$ is generated by the elements of the form
\begin{align*}
&\langle y\rangle\langle a\rangle\langle b\rangle\langle x\rangle
-\langle y\rangle\langle a,b\rangle\langle x\rangle
+\langle y\rangle\langle b,a*b\rangle\langle x\rangle, \\
&\langle y\rangle\langle x\rangle\langle b\rangle\langle c\rangle
-\langle y\rangle\langle x\rangle\langle b,c\rangle
+\langle y\rangle\langle x\rangle\langle c,b*c\rangle, \\
&\langle y\rangle\langle a,b\rangle\langle c\rangle
-\langle y\rangle\langle a,b,c\rangle
+\langle y\rangle\langle a,c,b*c\rangle
-\langle y\rangle\langle c,a*c,b*c\rangle, \\
&\langle y\rangle\langle a\rangle\langle b,c\rangle
-\langle y\rangle\langle a,b,c\rangle
+\langle y\rangle\langle b,a*b,c\rangle
-\langle y\rangle\langle b,c,a*(bc)\rangle
\end{align*}
for $a,b,c \in G_{\lambda},x\in X$.

\begin{definition}
The quotient complex of $P_*(X)_Y$ modulo decomposition degeneracies $D_*(X)_Y$ is denoted by 
 $C_*(X)_Y=(C_n(X)_Y,\partial_n)$, where $C_n(X)_Y=P_n(X)_Y/D_n(X)_Y$. 
For an abelian group $A$, we define the cochain complex $C^*(X;A)_Y=\operatorname{Hom}(C_*(X)_Y,A)$.
We denote by $H_n(X)_Y$ the $n$th homology group of $C_*(X)_Y$.
\end{definition}

\section{Algebraic aspects of the homology}
\label{sec:alg}

In this section we study algebraic aspects of the homology theory we defined.
Specifically, we characterize the first homology group, and show that a $2$-cocycle defines an extension.
For simplicity we consider the case $Y=\{y_0\}$ is a singleton, and we suppress the symbols $\langle y_0\rangle$ whenever possible.

Let $X$ be a multiple conjugation quandle, and $Y=\{y_0\}$ be a singleton.
Then $P_0(X)_Y$ is infinite cyclic generated by $\langle y_0\rangle$, and
$\partial_1(\langle y_0\rangle\langle a\rangle)=\langle y_0*a\rangle- \langle y_0\rangle$ 
for $a \in X$.
Hence $H_0(X)_Y=\Z$.
Next we characterize $H_1(X)_Y$.

For any $x\in X$, $\lambda\in\Lambda$, there is a unique element $\mu\in\Lambda$ such that $S_x(G_\lambda)=G_\mu$ and that $S_x:G_\lambda\to G_\mu$ is a group isomorphism.
Then we define an equivalence relation on $\Lambda$ by $\lambda\sim\mu$ if and only if there exist $x_1,\ldots,x_n\in X$, $\varepsilon_1,\ldots,\varepsilon_n\in\{\pm1\}$ such that 
$(S_{x_n}^{\varepsilon {}_n}\circ\cdots\circ S_{x_1}^{\varepsilon {}_1})(G_\lambda)=G_\mu$.
The equivalence classes form a partition of $\Lambda$.
We call it the \textit{orbit partition}.

\begin{proposition}
Let $Y=\{y_0\}$ be a singleton.
Let $X=\coprod_{\lambda \in \Lambda}G_{\lambda}$ be a multiple conjugation quandle with a finite orbit partition $\Lambda=\bigcup_{i=1}^n\Lambda_i$.
Then the first homology group $H_1(X)_Y$ is isomorphic to $\prod_{i=1}^nG_{\lambda_i}^\mathrm{ab}$, the product of abelianizations, where $\lambda_i\in\Lambda_i$.
\end{proposition}

\begin{proof}
We suppress the notation $\langle y_0\rangle$.
The $2$-dimensional chain group $P_2(X)_Y$ is generated by
\[ \{\langle a\rangle\langle b\rangle,\,\langle a,b\rangle\,|\,a,b\in X\}. \]
The degeneracy group $D_1(X)_Y$ is $0$.
One computes
\begin{align*}
&\partial_2(\langle a\rangle\langle b\rangle)
=-\langle a*b\rangle+\langle a\rangle,
&&\partial_2(\langle a,b\rangle)
=\langle b\rangle-\langle ab\rangle+\langle a\rangle.
\end{align*}
By definitions, $H_1(X)_Y$ is the free abelian group generated by
$\{\langle a\rangle\,|\,a\in X\}$ subject to the relations
$\langle a*b\rangle=\langle a\rangle$ and
$\langle ab\rangle=\langle a\rangle+\langle b\rangle$ over all $a,b\in X$.
The elements of the form $\langle a\rangle$ from the groups
$\{G_\lambda\,|\,\lambda\in\Lambda_i\}$ are identified by $S_b$ from the first relation.
The last relation gives rise to $G_\lambda^\mathrm{ab}$ for each subgroup generated by elements of $G_\lambda$.
Then the result follows.
\end{proof}

\begin{proposition}
Let $X=\coprod_{\lambda\in\Lambda}G_\lambda$ be a multiple conjugation quandle, let $Y=\{y_0\}$ be a singleton, and $A$ an abelian group.
A map $\phi:P_2(X)_Y\to A$ is a $2$-cocycle of $C^*(X)_Y$ if and only if
$X\times A=\coprod_{\lambda\in\Lambda} (G_\lambda\times A) $ with  
\begin{align*}
&(a,s)*(b,t):=(a*b,s+\phi(\langle a\rangle\langle b\rangle))
\text{ for $(a,s),(b,t)\in X\times A$}, \\
&(a,s)(b,t):=(ab,s+t+\phi(\langle a,b\rangle))
\text{ for $(a,s),(b,t)\in G_\lambda\times A$}
\end{align*}
is a multiple conjugation quandle, where
$\phi(\langle y_0\rangle\langle a\rangle\langle b\rangle)$ and
$\phi(\langle y_0\rangle\langle a,b\rangle)$ are respectively denoted by
$\phi(\langle a\rangle\langle b\rangle)$ and
$\phi(\langle a,b\rangle)$ for short.
Furthermore, $(e_\lambda,-\phi(\langle e_\lambda,e_\lambda\rangle))$ is the identity of the group $G_\lambda\times A$, and
$(a^{-1},-s-\phi(\langle a,a^{-1}\rangle)
-\phi(\langle e_\lambda,e_\lambda\rangle))$
is the inverse of $(a,s)\in G_\lambda\times A$.
\end{proposition}

\begin{proof}
We show correspondences between cocycle conditions and MCQ conditions for the extension.

\noindent
(1)  
{\it The correspondence between the cocycle condition $\phi(\partial_3(\langle a,b,c\rangle))=0$
and the associativity of a group.}

For $(a,s),(b,t),(c,u)\in G_\lambda\times A$,
$\phi(\langle a,b\rangle)+\phi(\langle ab,c\rangle)
=\phi(\langle b,c\rangle)+\phi(\langle a,bc\rangle)$ if and only if
$((a,s)(b,t))(c,u)=(a,s)((b,t)(c,u))$, since
\begin{align*}
&((a,s)(b,t))(c,u)
=(abc,s+t+u+\phi(\langle a,b\rangle)+\phi(\langle ab,c\rangle)), \\
&(a,s)((b,t)(c,u))
=(abc,s+t+u+\phi(\langle b,c\rangle)+\phi(\langle a,bc\rangle)).
\end{align*}
We note that
$\phi(\langle a,b\rangle)+\phi(\langle ab,c\rangle)
=\phi(\langle b,c\rangle)+\phi(\langle a,bc\rangle)$, or equivalently
$((a,s)(b,t))(c,u)=(a,s)((b,t)(c,u))$ implies that
$\phi(\langle a,e_\lambda\rangle)=\phi(\langle e_\lambda,c\rangle)$,
$\phi(\langle b^{-1},b\rangle)=\phi(\langle b,b^{-1}\rangle)$.
These equalities respectively imply
\[ (a,s)
=(a,s)(e_\lambda,-\phi(\langle e_\lambda,e_\lambda\rangle))
=(e_\lambda,-\phi(\langle e_\lambda,e_\lambda\rangle))(a,s) \]
and
\begin{align*}
(e_\lambda,-\phi(\langle e_\lambda,e_\lambda\rangle))
&=(a,s)(a^{-1},-s-\phi(\langle a,a^{-1}\rangle)
-\phi(\langle e_\lambda,e_\lambda\rangle)) \\
&=(a^{-1},-s-\phi(\langle a,a^{-1}\rangle)
-\phi(\langle e_\lambda,e_\lambda\rangle))(a,s).
\end{align*}
It follows that $(e_\lambda,-\phi(\langle e_\lambda,e_\lambda\rangle))$ is the identity of the group $G_\lambda\times A$, and that
$(a^{-1},-s-\phi(\langle a,a^{-1}\rangle)
-\phi(\langle e_\lambda,e_\lambda\rangle))$ is the inverse of $(a,s)\in G_\lambda\times A$.

\noindent
(2)  
{\it The correspondence between the degeneracy of $\phi$ on $D_2(X)_Y$ and the first axiom of MCQ.}

For $(a,s),(b,t)\in G_\lambda\times A$,
$\phi(\langle a\rangle\langle b\rangle)+\phi(\langle b,a*b\rangle)
=\phi(\langle a,b\rangle)$ if and only if
$(b,t)((a,s)*(b,t))=(a,s)(b,t)$, since
\begin{align*}
&(b,t)((a,s)*(b,t))
=(b(a*b),s+t+\phi(\langle a\rangle\langle b\rangle)
+\phi(\langle b,a*b\rangle)), \\
&(a,s)(b,t)=(ab,s+t+\phi(\langle a,b\rangle)).
\end{align*}

\noindent
(3)  
{\it The correspondence between the cocycle condition $\phi(\partial_3(\langle x\rangle\langle a,b\rangle))=0$
and the second axiom of MCQ.}

For $(x,r)\in X\times A$, $(a,s),(b,t)\in G_\lambda\times A$,
$\phi(\langle x\rangle\langle ab\rangle)
=\phi(\langle x\rangle\langle a\rangle)
+\phi(\langle x*a\rangle\langle b\rangle)$ if and only if
$(x,r)*((a,s)(b,t))=((x,r)*(a,s))*(b,t)$, since
\begin{align*}
&(x,r)*((a,s)(b,t))
=(x*(ab),r+\phi(\langle x\rangle\langle ab\rangle)), \\
&((x,r)*(a,s))*(b,t)
=((x*a)*b,r+\phi(\langle x\rangle\langle a\rangle)
+\phi(\langle x*a\rangle\langle b\rangle)).
\end{align*}
We note that
$\phi(\langle x\rangle\langle ab\rangle)
=\phi(\langle x\rangle\langle a\rangle)
+\phi(\langle x*a\rangle\langle b\rangle)$, or equivalently
$(x,r)*((a,s)(b,t))=((x,r)*(a,s))*(b,t)$ implies that
$\phi(\langle x\rangle\langle e_\lambda\rangle)=0$.
Then we have
\[ (a,s)*(e_\lambda,-\phi(\langle e_\lambda,e_\lambda\rangle))=(a,s). \]

\noindent
(4) 
{\it The correspondence between the cocycle condition $\phi(\partial_3(\langle a\rangle\langle b\rangle\langle c\rangle))=0$
and the third axiom of MCQ.}

For $(a,s),(b,t),(c,u)\in X\times A$,
$\phi(\langle a\rangle\langle b\rangle)
+\phi(\langle a*b\rangle\langle c\rangle)
=\phi(\langle a\rangle\langle c\rangle)
+\phi(\langle a*c\rangle\langle b*c\rangle)$ if and only if
$((a,s)*(b,t))*(c,u)=((a,s)*(c,u))*((b,t)*(c,u))$, since
\begin{align*}
&((a,s)*(b,t))*(c,u)
=((a*b)*c,s+\phi(\langle a\rangle\langle b\rangle)
+\phi(\langle a*b\rangle\langle c\rangle)), \\
&((a,s)*(c,u))*((b,t)*(c,u))
=((a*c)*(b*c),s+\phi(\langle a\rangle\langle c\rangle)
+\phi(\langle a*c\rangle\langle b*c\rangle)).
\end{align*}

\noindent
(5) 
{\it  The correspondence between the cocycle condition $\phi(\partial_3(\langle a,b\rangle\langle x\rangle))=0$
and the last axiom of MCQ.}

For $(x,r)\in X\times A$, $(a,s),(b,t)\in G_\lambda\times A$,
$\phi(\langle a,b\rangle)
+\phi(\langle ab\rangle\langle x\rangle)
=\phi(\langle a\rangle\langle x\rangle)
+\phi(\langle b\rangle\langle x\rangle)
+\phi(\langle a*x,b*x\rangle)$ if and only if
$((a,s)(b,t))*(x,r)=((a,s)*(x,r))((b,t)*(x,r))$, since
\begin{align*}
&((a,s)(b,t))*(x,r)=((ab)*x,s+t+\phi(\langle a,b\rangle)
+\phi(\langle ab\rangle\langle x\rangle)), \\
&((a,s)*(x,r))((b,t)*(x,r)) \\
&\hspace{10mm}=((a*x)(b*x),s+t+\phi(\langle a\rangle\langle x\rangle)
+\phi(\langle b\rangle\langle x\rangle)
+\phi(\langle a*x,b*x\rangle)).
\end{align*}

Therefore $\phi$ is a $2$-cocycle if and only if $X\times A$ is a multiple conjugation quandle.
\end{proof}

\section{Quandle cocycle invariants for handlebody-links}
\label{sec:coloring}

The definition of a multiple conjugation quandle is motivated from handlebody-links and their colorings~\cite{Ishii_MCQ}. 
A \textit{handlebody-link} is a disjoint union of handlebodies embedded in the $3$-sphere $S^3$.
A \textit{handlebody-knot} is a one component handlebody-link.
Two handlebody-links are \textit{equivalent} if there is an orientation-preserving self-homeomorphism of $S^3$ which sends one to the other.
A \textit{diagram} of a handlebody-link is a diagram of a spatial trivalent graph whose regular neighborhood is the handlebody-link, where a spatial trivalent graph is a finite trivalent graph embedded in $S^3$.
In this paper, a trivalent graph may contain circle components.
Two handlebody-links are equivalent if and only if their diagrams are related by a finite sequence of R1--R6 moves depicted in Figure~\ref{fig:Rmove}~\cite{Ishii08}.

\begin{figure}
\begin{center}
\includegraphics{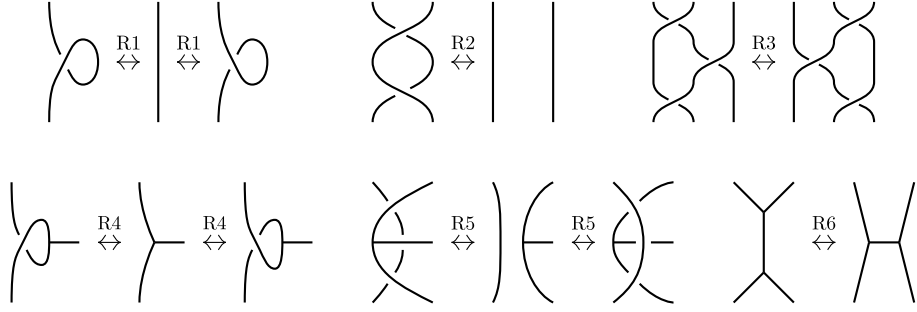}
\end{center}
\caption{Reidemeister moves for handlebody-links}
\label{fig:Rmove}
\end{figure}

An \textit{$S^1$-orientation} of a handlebody-link is an orientation of all genus $1$ components of the handlebody-link, where an orientation of a solid torus is an orientation of its core $S^1$.
Two $S^1$-oriented handlebody-links are \textit{equivalent} if there is an orientation-preserving self-homeomorphism of $S^3$ which sends one to the other preserving the $S^1$-orientation.
A \textit{Y-orientation} of a spatial trivalent graph is an orientation of the graph without sources and sinks with respect to the orientation (see Figure~\ref{fig:Y-orientation}).
We note that the term Y-orientation is a symbolic convention, and has no relation to  an $X$-set $Y$.
A \textit{diagram} of an $S^1$-oriented handlebody-link is a diagram of a Y-oriented spatial trivalent graph whose regular neighborhood is the $S^1$-oriented handlebody-link where the $S^1$-orientation is induced from the Y-orientation by forgetting the orientations except on circle components of the Y-oriented spatial trivalent graph.
\textit{Y-oriented R1--R6 moves} are R1--R6 moves between two diagrams with Y-orientations which are identical except in the disk where the move applied.
Two $S^1$-oriented handlebody-links are equivalent if and only if their diagrams are related by a finite sequence of Y-oriented R1--R6 moves~\cite{Ishii16}.
Note that in Figure~\ref{fig:Rmove} (R6), if all end points are oriented downward, then either choice of the two possible orientations of the middle edge makes the diagram Y-oriented locally.
Thus reversing an orientation of this edge can be regarded as applying Y-oriented R6 moves twice.
This is the case whenever both orientations of an edge give Y-orientations.

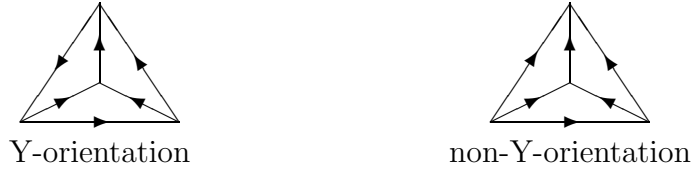
\begin{figure}
\mbox{}\hfill
\begin{picture}(60,60)
\put(0,15){\line(1,0){60}}
\put(0,15){\line(2,1){30}}
\put(0,15){\line(2,3){30}}
\put(60,15){\line(-2,1){30}}
\put(60,15){\line(-2,3){30}}
\put(30,30){\line(0,1){30}}
\thicklines
\put(34,15){\vector(1,0){0}}
\put(19,24.5){\vector(2,1){0}}
\put(13,34.5){\vector(-2,-3){0}}
\put(41,24.5){\vector(-2,1){0}}
\put(43,40.5){\vector(-2,3){0}}
\put(30,48){\vector(0,1){0}}
\put(30,4){\makebox(0,0){Y-orientation}}
\end{picture}
\hfill
\begin{picture}(60,60)
\put(0,15){\line(1,0){60}}
\put(0,15){\line(2,1){30}}
\put(0,15){\line(2,3){30}}
\put(60,15){\line(-2,1){30}}
\put(60,15){\line(-2,3){30}}
\put(30,30){\line(0,1){30}}
\thicklines
\put(34,15){\vector(1,0){0}}
\put(19,24.5){\vector(2,1){0}}
\put(18,42){\vector(2,3){0}}
\put(41,24.5){\vector(-2,1){0}}
\put(43,40.5){\vector(-2,3){0}}
\put(30,48){\vector(0,1){0}}
\put(30,4){\makebox(0,0){non-Y-orientation}}
\end{picture}
\hfill\mbox{}
\caption{Y-orientation}
\label{fig:Y-orientation}
\end{figure}

Let $X=\coprod_{\lambda\in\Lambda}G_\lambda$ be a multiple conjugation quandle, and let $Y$ be an $X$-set.
Let $D$ be a diagram of an $S^1$-oriented handlebody-link $H$.
We denote by $\mathcal{A}(D)$ the set of arcs of $D$, where an arc is a piece of a curve each of whose endpoints is an undercrossing or a vertex.
We denote by $\mathcal{R}(D)$ the set of complementary regions of $D$.
In this paper, an orientation of an arc is represented by the normal orientation obtained by rotating the usual orientation counterclockwise by $\pi/2$ on the diagram.
An \textit{$X$-coloring} $C$ of a diagram $D$ is an assignment of an element of $X$ to each arc $\alpha\in\mathcal{A}(D)$ satisfying the conditions depicted in the left three diagrams in Figure~\ref{fig:coloring} at each crossing and each vertex of $D$.
An $X_Y$-coloring $C$ of $D$ is an extension of an $X$-coloring of $D$ which assigns an element of $Y$ to each region $R\in\mathcal{R}(D)$ satisfying the condition depicted in the rightmost diagram in Figure~\ref{fig:coloring} at each arc.
We denote by $\operatorname{Col}_X(D)$ (resp.~$\operatorname{Col}_X(D)_Y$) the set of $X$-colorings (resp.~$X_Y$-colorings) of $D$.
Then we have the following proposition.

\begin{figure}
\mbox{}\hfill
\begin{picture}(70,80)
 \put(10,50){\line(1,0){17}}
 \put(50,50){\line(-1,0){17}}
 \put(30,30){\line(0,1){40}}
 \put(31,65){\makebox(0,0){$\rightarrow$}}
 \put(30,73){\makebox(0,0)[b]{$b$}}
 \put(7,50){\makebox(0,0)[r]{$a$}}
 \put(53,50){\makebox(0,0)[l]{$a*b$}}
\thicklines
 \put(30,30){\vector(0,-1){0}}
\end{picture}
\hfill
\begin{picture}(60,80)
 \put(10,70){\line(1,-1){20}}
 \put(50,70){\line(-1,-1){20}}
 \put(30,30){\line(0,1){20}}
 \put(15,65){\makebox(0,0){$\nearrow$}}
 \put(45,65){\makebox(0,0){$\searrow$}}
 \put(31,35){\makebox(0,0){$\rightarrow$}}
 \put(8,72){\makebox(0,0)[br]{$a$}}
 \put(52,72){\makebox(0,0)[bl]{$b$}}
 \put(30,27){\makebox(0,0)[t]{$ab$}}
 \put(30,5){\makebox(0,0){$a,b\in G_\lambda$}}
\end{picture}
\hfill
\begin{picture}(60,80)
 \put(10,30){\line(1,1){20}}
 \put(50,30){\line(-1,1){20}}
 \put(30,70){\line(0,-1){20}}
 \put(15,35){\makebox(0,0){$\searrow$}}
 \put(45,35){\makebox(0,0){$\nearrow$}}
 \put(31,65){\makebox(0,0){$\rightarrow$}}
 \put(8,28){\makebox(0,0)[tr]{$a$}}
 \put(52,28){\makebox(0,0)[tl]{$b$}}
 \put(30,73){\makebox(0,0)[b]{$ab$}}
 \put(30,5){\makebox(0,0){$a,b\in G_\lambda$}}
\end{picture}
\hfill
\begin{picture}(70,80)
 \put(30,30){\line(0,1){40}}
 \put(31,65){\makebox(0,0){$\rightarrow$}}
 \put(30,73){\makebox(0,0)[b]{$a$}}
 \put(10,45){\framebox(10,10){$x$}}
 \put(40,45){\framebox(25,10){$x*a$}}
\thicklines
 \put(30,30){\vector(0,-1){0}}
\end{picture}
\hfill\mbox{}
\caption{Rules of a coloring}
\label{fig:coloring}
\end{figure}
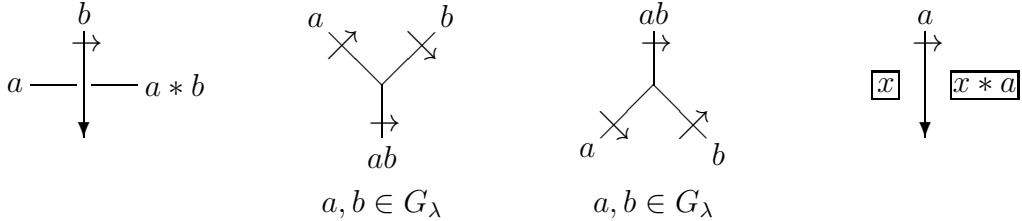

\begin{proposition}[\cite{Ishii16}]  \label{prop:coloring}
Let $X=\coprod_{\lambda\in\Lambda}G_\lambda$ be a multiple conjugation quandle, and let $Y$ be an $X$-set.
Let $D$ be a diagram of an $S^1$-oriented handlebody-link $H$.
Let $D'$ be a diagram obtained by applying one of Y-oriented R1--R6 moves to the diagram $D$ once.
For an $X$-coloring (resp.~$X_Y$-coloring) $C$ of $D$, there is a unique $X$-coloring (resp.~$X_Y$-coloring) $C'$ of $D'$ which coincides with $C$ except near a point where the move applied.
\end{proposition}

For an $X_Y$-coloring $C$ of a diagram $D$ of an $S^1$-oriented handlebody-link, 
we define the \textit{local chain}s 
$w(\xi;C)\in C_2(X)_Y$ 
at each crossing $\xi$ and each vertex $\xi$ of $D$ as depicted in Figure~\ref{fig:weights}.
We define a chain $W(D;C)\in C_2(X)_Y$ by
\[ W(D;C)=\sum_\xi w(\xi;C), \]
where $\xi$ runs over all crossings and vertices of $D$.
This is similar to the definitions found in \cite{CarterKamadaSaito01} for links and surface-links, and 
in \cite{IshiiIwakiri12} for handlebody-links. 

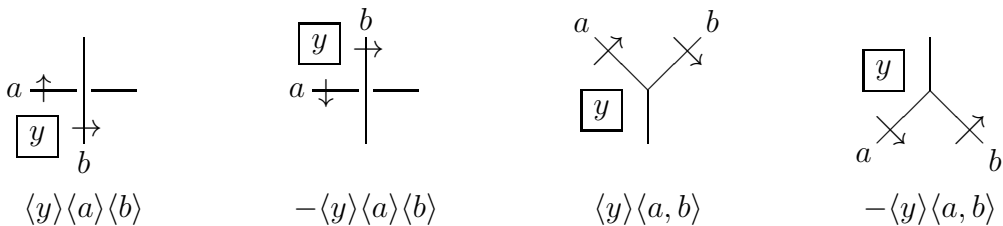
\begin{figure}[htb]
\mbox{}\hfill
\begin{picture}(60,80)
 \put(10,50){\line(1,0){17}}
 \put(50,50){\line(-1,0){17}}
 \put(30,30){\line(0,1){40}}
 \put(15,51){\makebox(0,0){$\uparrow$}}
 \put(31,35){\makebox(0,0){$\rightarrow$}}
 \put(7,50){\makebox(0,0)[r]{$a$}}
 \put(30,27){\makebox(0,0)[t]{$b$}}
 \put(5,25){\framebox(15,15){$y$}}
 \put(30,5){\makebox(0,0){$\langle y\rangle\langle a\rangle\langle b\rangle$}}
\end{picture}
\hfill
\begin{picture}(60,80)
 \put(10,50){\line(1,0){17}}
 \put(50,50){\line(-1,0){17}}
 \put(30,30){\line(0,1){40}}
 \put(15,49){\makebox(0,0){$\downarrow$}}
 \put(31,65){\makebox(0,0){$\rightarrow$}}
 \put(7,50){\makebox(0,0)[r]{$a$}}
 \put(30,73){\makebox(0,0)[b]{$b$}}
 \put(5,60){\framebox(15,15){$y$}}
 \put(30,5){\makebox(0,0){$-\langle y\rangle\langle a\rangle\langle b\rangle$}}
\end{picture}
\hfill
\begin{picture}(60,80)
 \put(10,70){\line(1,-1){20}}
 \put(50,70){\line(-1,-1){20}}
 \put(30,30){\line(0,1){20}}
 \put(15,65){\makebox(0,0){$\nearrow$}}
 \put(45,65){\makebox(0,0){$\searrow$}}
 \put(8,72){\makebox(0,0)[br]{$a$}}
 \put(52,72){\makebox(0,0)[bl]{$b$}}
 \put(5,35){\framebox(15,15){$y$}}
 \put(30,5){\makebox(0,0){$\langle y\rangle\langle a,b\rangle$}}
\end{picture}
\hfill
\begin{picture}(60,80)
 \put(10,30){\line(1,1){20}}
 \put(50,30){\line(-1,1){20}}
 \put(30,70){\line(0,-1){20}}
 \put(15,35){\makebox(0,0){$\searrow$}}
 \put(45,35){\makebox(0,0){$\nearrow$}}
 \put(8,28){\makebox(0,0)[tr]{$a$}}
 \put(52,28){\makebox(0,0)[tl]{$b$}}
 \put(5,50){\framebox(15,15){$y$}}
 \put(30,5){\makebox(0,0){$-\langle y\rangle\langle a,b\rangle$}}
\end{picture}
\hfill\mbox{}
\caption{Local chains represented by crossings and vertices}
\label{fig:weights}
\end{figure}

\begin{lemma} \label{lem:W(D;C)_2-cycle}
The chain $W(D;C)$ is a $2$-cycle of $C_*(X)_Y$.
Further, for cohomologous $2$-cocycles $\theta,\theta'$ of $C^*(X;A)_Y$,
we have $\theta(W(D;C))=\theta'(W(D;C))$.
\end{lemma}

\begin{proof}
It is sufficient to show that $W(D;C)$ is a $2$-cycle of $C_*(X)_Y$.
We denote by $\mathcal{SA}(D)$ the set of semi-arcs of $D$, where a semi-arc is a piece of a curve each of whose endpoints is a crossing or a vertex.
We denote by $\mathcal{SA}(D;\xi)$ the set of semi-arcs incident to $\xi$, where $\xi$ is a crossing or a vertex of $D$.

For a semi-arc $\alpha$, there is a unique region $R_\alpha$ facing $\alpha$ such that the normal orientation of $\alpha$ points from the region $R_\alpha$ to the opposite region with respect to $\alpha$.
For a semi-arc $\alpha$ incident to a crossing or a vertex $\xi$, we
define
\[ \epsilon(\alpha;\xi):=\begin{cases}
1 & \text{if the orientation of $\alpha$ points to $\xi$,} \\
-1 & \text{otherwise.}
\end{cases} \]

Let $\chi_1,\ldots,\chi_4$ and $\omega_1,\omega_2,\omega_3$ be respectively the semi-arcs incident to a crossing $\chi$ and a vertex $\omega$ as depicted in Figure~\ref{fig:semi-arcs}.
From
\begin{align*}
&\partial_2(w(\chi;C))=\sum_{\alpha\in\mathcal{SA}(D;\chi)}\epsilon(\alpha;\chi)\langle C(R_{\alpha})\rangle\langle C(\alpha)\rangle, \\
&\partial_2(w(\omega;C))=\sum_{\alpha\in\mathcal{SA}(D;\omega)}\epsilon(\alpha;\omega)\langle C(R_{\alpha})\rangle\langle C(\alpha)\rangle,
\end{align*}
it follows that
\[ \partial_2(W(D;C))
=\sum_\chi\partial_2(w(\chi;C))
+\sum_\omega\partial_2(w(\omega;C))=0, \]
where $\chi$ and $\omega$ respectively run over all crossings and vertices of $D$.
\end{proof}

\begin{figure}[htb]
\mbox{}\hfill
\begin{picture}(60,60)
 \put(10,30){\line(1,0){17}}
 \put(50,30){\line(-1,0){17}}
 \put(30,10){\line(0,1){40}}
 \put(15,31){\makebox(0,0){$\uparrow$}}
 \put(31,15){\makebox(0,0){$\rightarrow$}}
 \put(7,30){\makebox(0,0)[r]{$\chi_1$}}
 \put(53,30){\makebox(0,0)[l]{$\chi_2$}}
 \put(30,53){\makebox(0,0)[b]{$\chi_4$}}
 \put(30,7){\makebox(0,0)[t]{$\chi_3$}}
\end{picture}
\hfill
\begin{picture}(60,60)
 \put(10,30){\line(1,0){17}}
 \put(50,30){\line(-1,0){17}}
 \put(30,10){\line(0,1){40}}
 \put(15,29){\makebox(0,0){$\downarrow$}}
 \put(31,45){\makebox(0,0){$\rightarrow$}}
 \put(7,30){\makebox(0,0)[r]{$\chi_1$}}
 \put(53,30){\makebox(0,0)[l]{$\chi_2$}}
 \put(30,53){\makebox(0,0)[b]{$\chi_3$}}
 \put(30,7){\makebox(0,0)[t]{$\chi_4$}}
\end{picture}
\hfill
\begin{picture}(60,60)
 \put(10,50){\line(1,-1){20}}
 \put(50,50){\line(-1,-1){20}}
 \put(30,10){\line(0,1){20}}
 \put(15,45){\makebox(0,0){$\nearrow$}}
 \put(45,45){\makebox(0,0){$\searrow$}}
 \put(31,15){\makebox(0,0){$\rightarrow$}}
 \put(8,52){\makebox(0,0)[br]{$\omega_1$}}
 \put(52,52){\makebox(0,0)[bl]{$\omega_2$}}
 \put(30,5){\makebox(0,0)[t]{$\omega_3$}}
\end{picture}
\hfill
\begin{picture}(60,60)
 \put(10,10){\line(1,1){20}}
 \put(50,10){\line(-1,1){20}}
 \put(30,50){\line(0,-1){20}}
 \put(15,15){\makebox(0,0){$\searrow$}}
 \put(45,15){\makebox(0,0){$\nearrow$}}
 \put(31,45){\makebox(0,0){$\rightarrow$}}
 \put(8,8){\makebox(0,0)[tr]{$\omega_1$}}
 \put(52,8){\makebox(0,0)[tl]{$\omega_2$}}
 \put(30,55){\makebox(0,0)[b]{$\omega_3$}}
\end{picture}
\hfill\mbox{}
\caption{Semi-arcs near crossings and vertices} 
\label{fig:semi-arcs}
\end{figure}
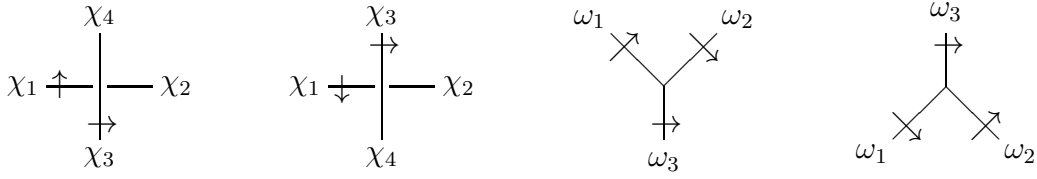

\begin{lemma} \label{lem:W(D;C)_Rmove}
Let $D$ be a diagram of an $S^1$-oriented handlebody-link $H$.
Let $D'$ be a diagram obtained by applying one of Y-oriented R1--R6 moves to the diagram $D$ once.
Let $C$ be an $X_Y$-coloring of $D$, let $C'$ be the unique $X_Y$-coloring of $D'$ such that $C$ and $C'$ coincide except near a point where the move applied.
Then we have $[W(D;C)]=[W(D';C')]$ in $H_2(X)_Y$.
\end{lemma}

\begin{proof}
We have the invariance under the Y-oriented R1 and R4 moves, since the difference between $[W(D;C)]$ and $[W(D';C')]$ is an element of $D_2(X)_Y$.
The invariance under the Y-oriented R2 move follows from the signs of the crossings which appear in the move.
We have the invariance under the Y-oriented R3, R5, and R6 move, since the difference between $[W(D;C)]$ and $[W(D';C')]$ is an image of $\partial_3$.
See Figure~\ref{fig:weights on R6 moves} for Y-oriented R6 moves, where all arcs are directed from top to bottom.
\end{proof}

\begin{figure}[htb]
\[ \begin{array}{ccc@{\hspace{20mm}}ccc}
\begin{minipage}{47pt}
\begin{picture}(47,60)
\thicklines
\qbezier(23,42)(19,46)(15,50)
\qbezier(23,42)(27,46)(31,50)
\put(23,42){\line(0,-1){8}}
\put(47,34){\line(0,1){16}}
\put(35,22){\line(-1,1){12}}
\put(35,22){\line(1,1){12}}
\put(35,22){\line(0,-1){12}}
\put(15,53){\makebox(0,0)[b]{$a$}}
\put(31,53){\makebox(0,0)[b]{$b$}}
\put(47,53){\makebox(0,0)[b]{$c$}}
\put(25,28){\makebox(0,0)[r]{$ab$}}
\put(35,7){\makebox(0,0)[t]{$abc$}}
\put(0,25){\framebox(10,10){$y$}}
\end{picture}
\end{minipage}
& \leftrightarrow &
\begin{minipage}{47pt}
\begin{picture}(47,60)
\thicklines
\put(15,34){\line(0,1){16}}
\qbezier(39,42)(35,46)(31,50)
\qbezier(39,42)(43,46)(47,50)
\put(39,42){\line(0,-1){8}}
\put(27,22){\line(-1,1){12}}
\put(27,22){\line(1,1){12}}
\put(27,22){\line(0,-1){12}}
\put(15,53){\makebox(0,0)[b]{$a$}}
\put(31,53){\makebox(0,0)[b]{$b$}}
\put(47,53){\makebox(0,0)[b]{$c$}}
\put(37,27){\makebox(0,0)[l]{$bc$}}
\put(27,7){\makebox(0,0)[t]{$abc$}}
\put(0,25){\framebox(10,10){$y$}}
\end{picture}
\end{minipage}
&
\begin{minipage}{47pt}
\begin{picture}(47,60)
\thicklines
\put(35,38){\line(0,1){12}}
\put(35,38){\line(1,-1){12}}
\put(35,38){\line(-1,-1){12}}
\put(23,18){\line(0,1){8}}
\qbezier(23,18)(19,14)(15,10)
\qbezier(23,18)(27,14)(31,10)
\put(47,10){\line(0,1){16}}
\put(35,53){\makebox(0,0)[b]{$abc$}}
\put(24,31){\makebox(0,0)[r]{$ab$}}
\put(15,7){\makebox(0,0)[t]{$a$}}
\put(32,7){\makebox(0,0)[t]{$b$}}
\put(47,7){\makebox(0,0)[t]{$c$}}
\put(0,25){\framebox(10,10){$y$}}
\end{picture}
\end{minipage}
& \leftrightarrow &
\begin{minipage}{47pt}
\begin{picture}(47,60)
\thicklines
\put(27,38){\line(0,1){12}}
\put(27,38){\line(1,-1){12}}
\put(27,38){\line(-1,-1){12}}
\put(15,10){\line(0,1){16}}
\put(39,18){\line(0,1){8}}
\qbezier(39,18)(35,14)(31,10)
\qbezier(39,18)(43,14)(47,10)
\put(27,53){\makebox(0,0)[b]{$abc$}}
\put(38,32){\makebox(0,0)[l]{$bc$}}
\put(15,7){\makebox(0,0)[t]{$a$}}
\put(31,7){\makebox(0,0)[t]{$b$}}
\put(47,7){\makebox(0,0)[t]{$c$}}
\put(0,25){\framebox(10,10){$y$}}
\end{picture}
\end{minipage}
\vspace{5mm}\\
\begin{array}{@{}l@{}}
+\langle y\rangle\langle a,b\rangle \\
+\langle y\rangle\langle ab,c\rangle
\end{array}
&&\begin{array}{@{}l@{}}
+\langle y*a\rangle\langle b,c\rangle \\
+\langle y\rangle\langle a,bc\rangle
\end{array}
&
\begin{array}{@{}l@{}}
-\langle y\rangle\langle ab,c\rangle \\
-\langle y\rangle\langle a,b\rangle
\end{array}
&&\begin{array}{@{}l@{}}
-\langle y\rangle\langle a,bc\rangle \\
-\langle y*a\rangle\langle b,c\rangle
\end{array}
\vspace{10mm}\\
\begin{minipage}{52pt}
\begin{picture}(52,52)
\thicklines
\put(15,26){\line(0,1){16}}
\put(39,34){\line(0,1){8}}
\qbezier(39,34)(43,30)(47,26)
\qbezier(23,18)(19,22)(15,26)
\put(23,18){\line(1,1){16}}
\put(23,18){\line(0,-1){8}}
\put(47,10){\line(0,1){16}}
\put(15,45){\makebox(0,0)[b]{$a$}}
\put(39,45){\makebox(0,0)[b]{$bc$}}
\put(31,23){\makebox(0,0)[l]{$b$}}
\put(23,7){\makebox(0,0)[t]{$ab$}}
\put(47,7){\makebox(0,0)[t]{$c$}}
\put(0,21){\framebox(10,10){$y$}}
\end{picture}
\end{minipage}
& \leftrightarrow &
\begin{minipage}{40pt}
\begin{picture}(40,52)
\thicklines
\put(25,32){\line(-1,1){10}}
\put(25,32){\line(1,1){10}}
\put(25,20){\line(0,1){12}}
\put(25,20){\line(-1,-1){10}}
\put(25,20){\line(1,-1){10}}
\put(15,45){\makebox(0,0)[b]{$a$}}
\put(35,45){\makebox(0,0)[b]{$bc$}}
\put(26,26){\makebox(0,0)[l]{$abc$}}
\put(15,7){\makebox(0,0)[t]{$ab$}}
\put(35,7){\makebox(0,0)[t]{$c$}}
\put(0,21){\framebox(10,10){$y$}}
\end{picture}
\end{minipage}
&
\begin{minipage}{40pt}
\begin{picture}(40,52)
\thicklines
\put(25,32){\line(-1,1){10}}
\put(25,32){\line(1,1){10}}
\put(25,20){\line(0,1){12}}
\put(25,20){\line(-1,-1){10}}
\put(25,20){\line(1,-1){10}}
\put(15,45){\makebox(0,0)[b]{$ab$}}
\put(35,45){\makebox(0,0)[b]{$c$}}
\put(26,26){\makebox(0,0)[l]{$abc$}}
\put(15,7){\makebox(0,0)[t]{$a$}}
\put(35,7){\makebox(0,0)[t]{$bc$}}
\put(0,21){\framebox(10,10){$y$}}
\end{picture}
\end{minipage}
& \leftrightarrow &
\begin{minipage}{52pt}
\begin{picture}(52,52)
\thicklines
\put(23,34){\line(0,1){8}}
\qbezier(23,34)(19,30)(15,26)
\put(47,26){\line(0,1){16}}
\put(15,10){\line(0,1){16}}
\put(39,18){\line(-1,1){16}}
\qbezier(39,18)(43,22)(47,26)
\put(39,18){\line(0,-1){8}}
\put(24,45){\makebox(0,0)[b]{$ab$}}
\put(47,45){\makebox(0,0)[b]{$c$}}
\put(31,30){\makebox(0,0)[l]{$b$}}
\put(15,7){\makebox(0,0)[t]{$a$}}
\put(39,7){\makebox(0,0)[t]{$bc$}}
\put(0,21){\framebox(10,10){$y$}}
\end{picture}
\end{minipage}
\vspace{5mm}\\
\begin{array}{@{}l@{}}
-\langle y*a\rangle\langle b,c\rangle \\
+\langle y\rangle\langle a,b\rangle
\end{array}
&&\begin{array}{@{}l@{}}
+\langle y\rangle\langle a,bc\rangle \\
-\langle y\rangle\langle ab,c\rangle
\end{array}
&
\begin{array}{@{}l@{}}
\langle y\rangle\langle ab,c\rangle \\
-\langle y\rangle\langle a,bc\rangle
\end{array}
&&\begin{array}{@{}l@{}}
-\langle y\rangle\langle a,b\rangle \\
+\langle y*a\rangle\langle b,c\rangle
\end{array}
\end{array} \]
\caption{Chains for Y-oriented R6 moves}
\label{fig:weights on R6 moves}
\end{figure}
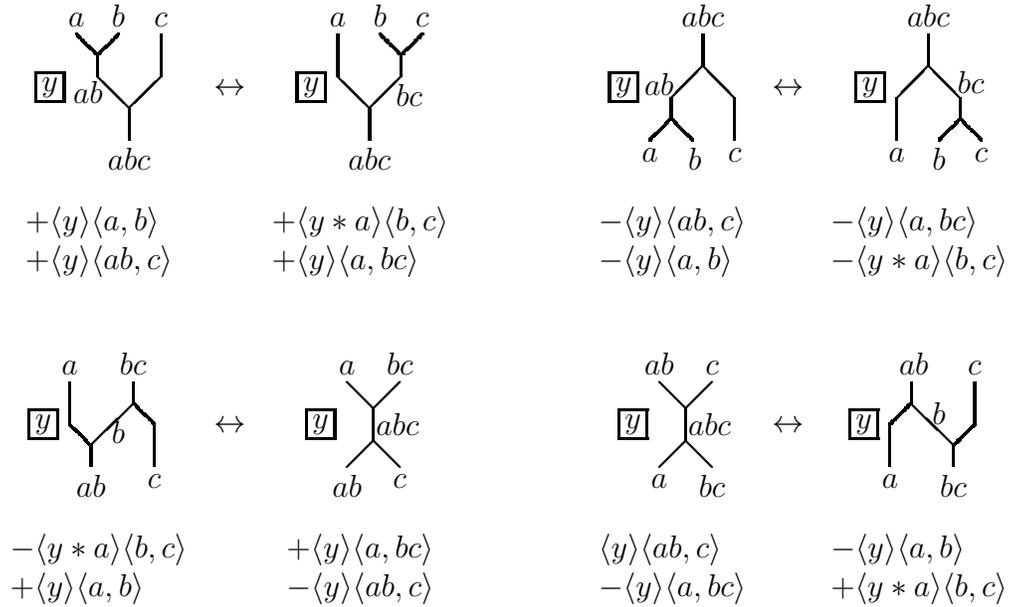

For a $2$-cocycle $\theta$ of $C^*(X;A)_Y$, we define
\begin{align*}
\mathcal{H}(D)
&:=\{[W(D;C)]\in H_2(X)_Y\,|\,C\in\operatorname{Col}_X(D)_Y\}, \\
\Phi_\theta(D)
&:=\{\theta(W(D;C))\in A\,|\,C\in\operatorname{Col}_X(D)_Y\}
\end{align*}
as multisets.
By Lemmas~\ref{lem:W(D;C)_2-cycle}, \ref{lem:W(D;C)_Rmove}, we have the following theorem.

\begin{theorem}
Let $D$ be a diagram of an $S^1$-oriented handlebody-link $H$.
Then $\mathcal{H}(D)$ and $\Phi_\theta(D)$ are invariants of $H$.
\end{theorem}

For an $S^1$-oriented handlebody-link $H$, 
let $H^*$ be the mirror image of $H$, and 
$-H$ be the $S^1$-oriented handleboby-link obtained from $H$ 
by reversing its $S^1$-orientation. 
Then we also have 
\begin{align*}
&\mathcal{H}(-H^*)=-\mathcal{H}(H),
&&\Phi_\theta(-H^*)=-\Phi_\theta(H) , 
\end{align*}
where $-S=\{-a\,|\,a\in S\}$ for a multiset $S$.
It is desirable to further study these invariants and applications to handlebody-links.

\section{For unoriented handlebody-links}
\label{sec:unori}

Let $X=\coprod_{\lambda\in\Lambda}G_\lambda$ be a multiple conjugation quandle, and let $Y$ be an $X$-set.
Let $D$ be a diagram of an (unoriented) handlebody-link $H$.
An $(X,\uparrow)$-color $C_\alpha$ of an arc $\alpha\in\mathcal{A}(D)$ is a map $C_\alpha$ from the set of orientations of the arc $\alpha$ to $X$ such that $C_\alpha(-o)=C_\alpha(o)^{-1}$, where $-o$ is the inverse of an orientation $o$.
An $(X,\uparrow)$-color $C_\alpha$ is represented by a pair of an orientation $o$ of $\alpha$ and an element $C_\alpha(o)\in X$ on the diagram $D$.
Two pairs $(o,a)$ and $(-o,a^{-1})$ represent the same $(X,\uparrow)$-color
(see Figure~\ref{fig:(o,a)=(-o,1/a)}).

\begin{figure}[htb]
\mbox{}\hfill
\begin{picture}(115,50)
\thicklines
\put(10,0){\line(0,1){40}}
\put(11,20){\makebox(0,0){$\rightarrow$}}
\put(18,15){\makebox(0,0){$a$}}
\put(50,20){\makebox(0,0){$=$}}
\put(90,0){\line(0,1){40}}
\put(89,20){\makebox(0,0){$\leftarrow$}}
\put(103,15){\makebox(0,0){$a^{-1}$}}
\end{picture}
\hfill\mbox{}
\caption{$(X,\uparrow)$-color}
\label{fig:(o,a)=(-o,1/a)}
\end{figure}

An \textit{$(X,\uparrow)$-coloring} $C$ of a diagram $D$ is an assignment of an $(X,\uparrow)$-color $C_\alpha$ to each arc $\alpha\in\mathcal{A}(D)$ satisfying the conditions depicted in the left two diagrams in Figure~\ref{fig:coloring(unoriented)} at each crossing and each vertex of $D$.
An $(X,\uparrow)_Y$-coloring $C$ of $D$ is an extension of an $(X,\uparrow)$-coloring of $D$ which assigns an element of $Y$ to each region $R\in\mathcal{R}(D)$ satisfying the condition depicted in the rightmost diagram in Figure~\ref{fig:coloring(unoriented)} at each arc.
We denote by $\operatorname{Col}_{(X,\uparrow)}(D)$ (resp.~$\operatorname{Col}_{(X,\uparrow)}(D)_Y$) the set of $(X,\uparrow)$-colorings (resp.~$(X,\uparrow)_Y$-colorings) of $D$.
The well-definedness of an $(X,\uparrow)$-coloring (resp.~$(X,\uparrow)_Y$-coloring) follows from
\begin{align*}
&(a^{-1})^{-1}=a,
&&a^{-1}*b=(a*b)^{-1},
&&(a*b)*b^{-1}=a, \\
&b(ab)^{-1}=a^{-1},
&&(ab)^{-1}a=b^{-1}. 
\end{align*}
The first three equalities are the defining conditions of a {\it good involution}
considered by Kamada and Oshiro~\cite{Kamada07,KamadaOshiro10}.
They used the notion of a good-involution
precisely to allow for appropriate changes of orientations.
Following their arguments,  
we can show the following proposition in the same way as Proposition~\ref{prop:coloring}. 

\begin{figure}[htb]
\mbox{}\hfill
\begin{picture}(70,80)
 \put(10,50){\line(1,0){17}}
 \put(50,50){\line(-1,0){17}}
 \put(30,30){\line(0,1){40}}
 \put(15,49){\makebox(0,0){$\downarrow$}}
 \put(45,49){\makebox(0,0){$\downarrow$}}
 \put(31,65){\makebox(0,0){$\rightarrow$}}
 \put(30,73){\makebox(0,0)[b]{$b$}}
 \put(7,50){\makebox(0,0)[r]{$a$}}
 \put(53,50){\makebox(0,0)[l]{$a*b$}}
\thicklines
 \put(30,30){\vector(0,-1){0}}
\end{picture}
\hfill
\begin{picture}(60,80)
 \put(10,70){\line(1,-1){20}}
 \put(50,70){\line(-1,-1){20}}
 \put(30,30){\line(0,1){20}}
 \put(15,65){\makebox(0,0){$\nearrow$}}
 \put(45,65){\makebox(0,0){$\searrow$}}
 \put(31,35){\makebox(0,0){$\rightarrow$}}
 \put(8,72){\makebox(0,0)[br]{$a$}}
 \put(52,72){\makebox(0,0)[bl]{$b$}}
 \put(30,27){\makebox(0,0)[t]{$ab$}}
 \put(30,5){\makebox(0,0){$a,b\in G_\lambda$}}
\end{picture}
\hfill
\begin{picture}(70,80)
 \put(30,30){\line(0,1){40}}
 \put(31,65){\makebox(0,0){$\rightarrow$}}
 \put(30,73){\makebox(0,0)[b]{$a$}}
 \put(10,45){\framebox(10,10){$x$}}
 \put(40,45){\framebox(25,10){$x*a$}}
\thicklines
 \put(30,30){\vector(0,-1){0}}
\end{picture}
\hfill\mbox{}
\caption{Rules of an unoriented coloring}
\label{fig:coloring(unoriented)}
\end{figure}
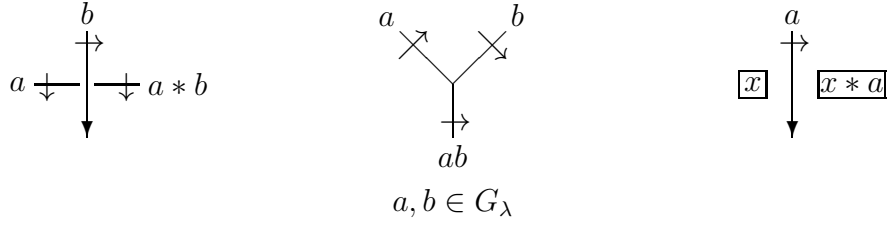

\begin{proposition}
Let $X=\coprod_{\lambda\in\Lambda}G_\lambda$ be a multiple conjugation quandle, and let $Y$ be an $X$-set.
Let $D$ be a diagram of a handlebody-link $H$.
Let $D'$ be a diagram obtained by applying one of the R1--R6 moves to the diagram $D$ once.
For an $(X,\uparrow)$-coloring (resp.~$(X,\uparrow)_Y$-coloring) $C$ of $D$, there is a unique $(X,\uparrow)$-coloring (resp.~$(X,\uparrow)_Y$-coloring) $C'$ of $D'$ which coincides with $C$ except near a point where the move applied.
\end{proposition}

Let $D_n^\uparrow(X)_Y$ be the subgroup of $P_n(X)_Y$ generated by the elements of the form
\[ \langle y\rangle\langle\boldsymbol{a}_1\rangle
\cdots\langle\boldsymbol{a}\rangle
\cdots\langle\boldsymbol{a}_k\rangle
+\langle y\rangle\langle\boldsymbol{a}_1\rangle
\cdots\langle\boldsymbol{a}\rangle_i^{-1}
\cdots\langle\boldsymbol{a}_k\rangle, \]
where $\langle a_1,\ldots,a_m\rangle_i^{-1}$ denotes
\[ \begin{cases}
*a_1\langle a_1^{-1},a_1a_2,a_3,\ldots,a_m\rangle
& \text{if $i=1$,} \\
\langle a_1,\ldots,a_{i-2},a_{i-1}a_i,a_i^{-1},a_ia_{i+1},a_{i+2},\ldots,a_m\rangle
& \text{if $i \neq1,m$,} \\
\langle a_1,\ldots,a_{m-2},a_{m-1}a_m,a_m^{-1}\rangle
& \text{if $i=m$.} \\ 
\end{cases} \]

The chain group $D^\uparrow_n(X)_Y$ will be called the \textit{group of orientation degeneracies}.
For example, $D^\uparrow_1(X)_Y$ is generated by the elements of the form
\[ \langle y\rangle\langle a\rangle+\langle y*a\rangle\langle a^{-1}\rangle, \]
and $D^\uparrow_2(X)_Y$ is generated by the elements of the form
\begin{align*}
&\langle y\rangle\langle a\rangle\langle b\rangle
+\langle y*a\rangle\langle a^{-1}\rangle\langle b\rangle,
&&\langle y\rangle\langle a\rangle\langle b\rangle
+\langle y*b\rangle\langle a*b\rangle\langle b^{-1}\rangle, \\
&\langle y\rangle\langle a,b\rangle
+\langle y*a\rangle\langle a^{-1},ab\rangle,
&&\langle y\rangle\langle a,b\rangle
+\langle y\rangle\langle ab,b^{-1}\rangle.
\end{align*}

We remark that the elements of the form
\begin{align*}
&\langle y\rangle\langle\boldsymbol{a}_1\rangle\cdots\langle a_1,\ldots,a_m\rangle
\cdots\langle\boldsymbol{a}_k\rangle \\
&\hspace{5mm}-(-1)^{m(m+1)/2}\langle y\rangle\langle\boldsymbol{a}_1\rangle\cdots
*(a_1\cdots a_m)\langle a_m^{-1},\ldots,a_1^{-1}\rangle
\cdots\langle\boldsymbol{a}_k\rangle
\end{align*}
belong to $D_n^\uparrow(X)_Y$.
Furthermore, we can prove that the elements of the form
\begin{align*}
&\langle y\rangle\langle\boldsymbol{a}_1\rangle\cdots
\langle a_1,\ldots,a_m\rangle\cdots\langle\boldsymbol{a}_k\rangle
-(-1)^{i(i+1)/2}\langle y\rangle\langle\boldsymbol{a}_1\rangle\cdots \\
&\hspace{5mm}*(a_1\cdots a_i)\langle a_i^{-1},\ldots,a_1^{-1},a_1\cdots a_{i+1},
a_{i+2},\ldots,a_m\rangle\cdots\langle\boldsymbol{a}_k\rangle
\end{align*}
belong to $D_n^\uparrow(X)_Y$ by induction.

\begin{lemma}
$D^\uparrow_*(X)_Y=(D^\uparrow_n(X)_Y,\partial_n)$ is a subcomplex of $P_*(X)_Y$.
\end{lemma}

\begin{proof}
We have $\partial_n(D^\uparrow_n(X)_Y)\subset D^\uparrow_{n-1}(X)_Y$, since
\begin{align*}
&\partial(\langle a_1,\ldots,a_m\rangle+*a_1\langle a_1^{-1},a_1a_2,a_3,\ldots,a_m\rangle) \\
&=\langle a_1,a_2a_3,a_4,\ldots,a_m\rangle+*a_1\langle a_1^{-1},a_1a_2a_3,a_4,\ldots,a_m\rangle \\
&\hspace{5mm}+\sum_{i=3}^{m-1}(-1)^i
(\langle a_1,\ldots,a_ia_{i+1},a_{i+2},\ldots,a_m\rangle \\
&\hspace{20mm}+*a_1\langle a_1^{-1},a_1a_2,a_2,\ldots,a_ia_{i+1},a_{i+2},\ldots,a_m\rangle) \\
&\hspace{5mm}+(-1)^m(\langle a_1,\ldots,a_{m-1}\rangle
+*a_1\langle a_1^{-1},a_1a_2,a_3,\ldots,a_{m-1}\rangle)
\end{align*}
and
\begin{align*}
&\partial(\langle a_1,\ldots,a_m\rangle
+\langle a_1,\ldots,a_{i-1}a_i,a_i^{-1},a_ia_{i+1},a_{i+2},\ldots,a_m\rangle) \\
&=*a_1\langle a_2,\ldots,a_m\rangle
+*a_1\langle a_2,\ldots,a_{i-1}a_i,a_i^{-1},a_ia_{i+1},a_{i+2},\ldots,a_m\rangle \\
&\hspace{5mm}+\sum_{j=1}^{i-2}(-1)^j(\langle a_1,\ldots,a_ja_{j+1},a_{j+2},\ldots,a_m\rangle \\
&\hspace{20mm}+\langle a_1,\ldots,a_ja_{j+1},a_{j+2},\ldots,a_{i-1}a_i,a_i^{-1},a_ia_{i+1},a_{i+2},\ldots,a_m\rangle) \\
&\hspace{5mm}+\sum_{j=i+1}^{m-1}(-1)^j(\langle a_1,\ldots,a_ja_{j+1},a_{j+2},\ldots,a_m\rangle \\
&\hspace{20mm}+\langle a_1,\ldots,a_{i-1}a_i,a_i^{-1},a_ia_{i+1},a_{i+2},\ldots,a_ja_{j+1},\ldots,a_m\rangle) \\
&\hspace{5mm}+(-1)^m(\langle a_1,\ldots,a_{m-1}\rangle
+\langle a_1,\ldots,a_{i-1}a_i,a_i^{-1},a_ia_{i+1},a_{i+2},\ldots,a_{m-1}\rangle).
\end{align*}
Thus $D^\uparrow_*(X)_Y$ is a subcomplex of $P_*(X)_Y$.
\end{proof}

\begin{definition}
We set $C^\uparrow_n(X)_Y=P_n(X)_Y/(D_n(X)_Y+D^\uparrow_n(X)_Y)$.
The quotient complex $(C^\uparrow_n(X)_Y,\partial_n)$ is denoted by $C^\uparrow_*(X)_Y.$
For an abelian group $A$, we define the cochain complex 
$C^*_\uparrow(X;A)_Y=\operatorname{Hom}(C^\uparrow_*(X)_Y,A)$.
We denote by $H^\uparrow_n(X)_Y$ the $n$th homology group of $C^\uparrow_*(X)_Y$. 
\end{definition}

For an $(X,\uparrow)_Y$-coloring $C$ of a diagram $D$ for a handlebody-link, 
we define the \textit{local chain}s $w(\xi;C)$ at each crossing $\xi$ and 
each vertex $\xi$ of $D$ as depicted in Figure~\ref{fig:weights}.
The local chain is well-defined, since
\begin{align*}
&-\langle y*a\rangle\langle a^{-1}\rangle\langle b\rangle
=\langle y\rangle\langle a\rangle\langle b\rangle
=-\langle y*b\rangle\langle a*b\rangle\langle b^{-1}\rangle, \\
&-\langle y*a\rangle\langle a^{-1},ab\rangle
=\langle y\rangle\langle a,b\rangle
=-\langle y\rangle\langle ab,b^{-1}\rangle
\end{align*}
in $C^\uparrow_2(X)_Y$ 
(see Figure~\ref{fig:well-defined weights}).
Then we can define the chain $W(D;C)\in C^\uparrow_2(X)_Y$ in the same way as $W(D;C)\in C_2(X)_Y$, and obtain invariants $\mathcal{H}(H)$, $\Phi_\theta(H)$ for an (unoriented) handlebody-link $H$.

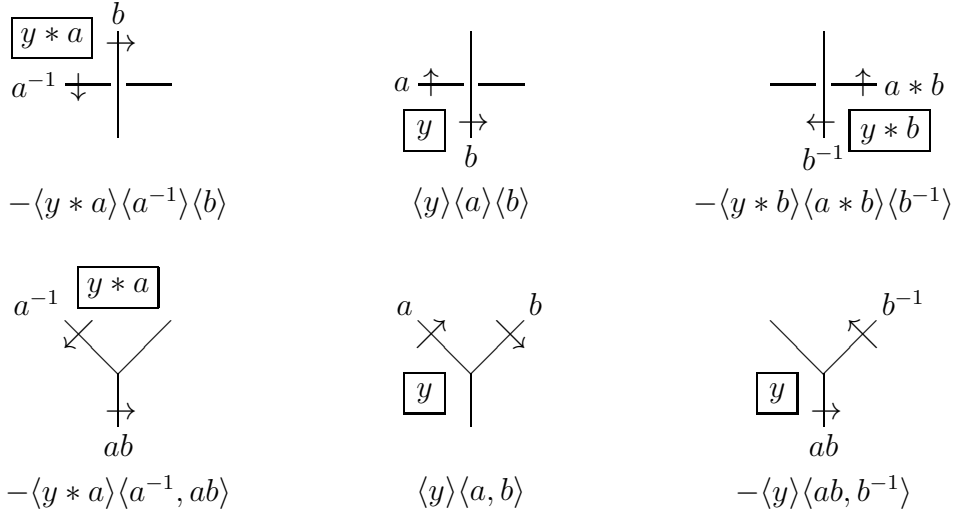
\begin{figure}[htb]
\mbox{}\hfill
\begin{picture}(60,80)
 \put(10,50){\line(1,0){17}}
 \put(50,50){\line(-1,0){17}}
 \put(30,30){\line(0,1){40}}
 \put(15,49){\makebox(0,0){$\downarrow$}}
 \put(31,65){\makebox(0,0){$\rightarrow$}}
 \put(7,50){\makebox(0,0)[r]{$a^{-1}$}}
 \put(30,73){\makebox(0,0)[b]{$b$}}
 \put(-10,60){\framebox(30,15){$y*a$}}
 \put(30,5){\makebox(0,0){$-\langle y*a\rangle\langle a^{-1}\rangle\langle b\rangle$}}
\end{picture}
\hfill
\begin{picture}(60,80)
 \put(10,50){\line(1,0){17}}
 \put(50,50){\line(-1,0){17}}
 \put(30,30){\line(0,1){40}}
 \put(15,51){\makebox(0,0){$\uparrow$}}
 \put(31,35){\makebox(0,0){$\rightarrow$}}
 \put(7,50){\makebox(0,0)[r]{$a$}}
 \put(30,27){\makebox(0,0)[t]{$b$}}
 \put(5,25){\framebox(15,15){$y$}}
 \put(30,5){\makebox(0,0){$\langle y\rangle\langle a\rangle\langle b\rangle$}}
\end{picture}
\hfill
\begin{picture}(60,80)
 \put(10,50){\line(1,0){17}}
 \put(50,50){\line(-1,0){17}}
 \put(30,30){\line(0,1){40}}
 \put(45,51){\makebox(0,0){$\uparrow$}}
 \put(29,35){\makebox(0,0){$\leftarrow$}}
 \put(53,50){\makebox(0,0)[l]{$a*b$}}
 \put(30,27){\makebox(0,0)[t]{$b^{-1}$}}
 \put(40,25){\framebox(30,15){$y*b$}}
 \put(30,5){\makebox(0,0){$-\langle y*b\rangle\langle a*b\rangle\langle b^{-1}\rangle$}}
\end{picture}
\hfill\mbox{}\vspace{10mm}\\
\mbox{}\hfill
\begin{picture}(60,80)
 \put(10,70){\line(1,-1){20}}
 \put(50,70){\line(-1,-1){20}}
 \put(30,30){\line(0,1){20}}
 \put(15,65){\makebox(0,0){$\swarrow$}}
 \put(31,35){\makebox(0,0){$\rightarrow$}}
 \put(8,72){\makebox(0,0)[br]{$a^{-1}$}}
 \put(30,27){\makebox(0,0)[t]{$ab$}}
 \put(15,75){\framebox(30,15){$y*a$}}
 \put(30,5){\makebox(0,0){$-\langle y*a\rangle\langle a^{-1},ab\rangle$}}
\end{picture}
\hfill
\begin{picture}(60,80)
 \put(10,70){\line(1,-1){20}}
 \put(50,70){\line(-1,-1){20}}
 \put(30,30){\line(0,1){20}}
 \put(15,65){\makebox(0,0){$\nearrow$}}
 \put(45,65){\makebox(0,0){$\searrow$}}
 \put(8,72){\makebox(0,0)[br]{$a$}}
 \put(52,72){\makebox(0,0)[bl]{$b$}}
 \put(5,35){\framebox(15,15){$y$}}
 \put(30,5){\makebox(0,0){$\langle y\rangle\langle a,b\rangle$}}
\end{picture}
\hfill
\begin{picture}(60,80)
 \put(10,70){\line(1,-1){20}}
 \put(50,70){\line(-1,-1){20}}
 \put(30,30){\line(0,1){20}}
 \put(45,65){\makebox(0,0){$\nwarrow$}}
 \put(31,35){\makebox(0,0){$\rightarrow$}}
 \put(52,72){\makebox(0,0)[bl]{$b^{-1}$}}
 \put(30,27){\makebox(0,0)[t]{$ab$}}
 \put(5,35){\framebox(15,15){$y$}}
 \put(30,5){\makebox(0,0){$-\langle y\rangle\langle ab,b^{-1}\rangle$}}
\end{picture}
\hfill\mbox{}
\caption{Well-definedness of local chains for unoriented handlebody-links} 
\label{fig:well-defined weights}
\end{figure}

\section{Simplicial decomposition}
\label{sec:decomp}

The goal of this section is to prove Lemma~\ref{lem:subcomplexM}
stating that $D_*(X)_Y$ is a subcomplex. 
The formula of $D_2(X)_Y$, when $\langle y\rangle$ is omitted, is written as
\[ \langle a\rangle\langle b\rangle-\langle a,b\rangle+\langle b,a*b\rangle, \]
and its geometric interpretation is depicted in Figure~\ref{dual}.
In (A), a colored triangle representing $\langle a,b\rangle$ is depicted, as well as its dual graph with a trivalent vertex.
The colorings of such a graph were discussed in Section~\ref{sec:coloring}.
A colored square representing $\langle a\rangle\langle b\rangle$ is depicted in (B), with the dual graph that corresponds to a crossing.
In (C), a triangulation of the square is depicted, and after triangulation it represents
$\langle a,b\rangle-\langle b,a*b\rangle$.
Thus the triangulation corresponds to the above formula.
This decomposition is found in \cite{CJKLS}.

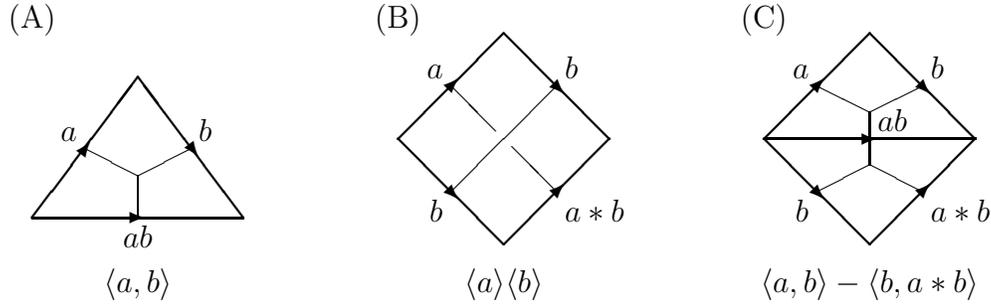
\begin{figure}[htb]
\mbox{}\hfill
\begin{picture}(80,110)
 \put(40,30){\line(0,1){16}}
 \put(40,46){\line(-2,1){20}}
 \put(40,46){\line(2,1){20}}
\thicklines
 \put(0,30){\line(3,4){40}}
 \put(0,30){\line(1,0){80}}
 \put(80,30){\line(-3,4){40}}
 \put(22,59){\vector(3,4){0}}
 \put(63,53){\vector(3,-4){0}}
 \put(43,30){\vector(1,0){0}}
 \put(17,59){\makebox(0,0)[br]{$a$}}
 \put(63,59){\makebox(0,0)[bl]{$b$}}
 \put(40,27){\makebox(0,0)[t]{$ab$}}
 \put(0,105){\makebox(0,0){(A)}}
 \put(40,5){\makebox(0,0){$\langle a,b\rangle$}}
\end{picture}
\hfill
\begin{picture}(80,110)
 \put(20,40){\line(1,1){40}}
 \put(60,40){\line(-1,1){17}}
 \put(20,80){\line(1,-1){17}}
\thicklines
 \put(40,20){\line(-1,1){40}}
 \put(40,20){\line(1,1){40}}
 \put(40,100){\line(-1,-1){40}}
 \put(40,100){\line(1,-1){40}}
 \put(23,37){\vector(1,-1){0}}
 \put(23,83){\vector(1,1){0}}
 \put(63,43){\vector(1,1){0}}
 \put(63,77){\vector(1,-1){0}}
 \put(17,83){\makebox(0,0)[br]{$a$}}
 \put(63,83){\makebox(0,0)[bl]{$b$}}
 \put(17,37){\makebox(0,0)[tr]{$b$}}
 \put(63,37){\makebox(0,0)[tl]{$a*b$}}
 \put(0,105){\makebox(0,0){(B)}}
 \put(40,5){\makebox(0,0){$\langle a\rangle\langle b\rangle$}}
\end{picture}
\hfill
\begin{picture}(80,110)
 \put(40,70){\line(-2,1){20}}
 \put(40,70){\line(2,1){20}}
 \put(40,50){\line(0,1){20}}
 \put(40,50){\line(-2,-1){20}}
 \put(40,50){\line(2,-1){20}}
\thicklines
 \put(40,20){\line(-1,1){40}}
 \put(40,20){\line(1,1){40}}
 \put(0,60){\line(1,0){80}}
 \put(40,100){\line(-1,-1){40}}
 \put(40,100){\line(1,-1){40}}
 \put(23,37){\vector(1,-1){0}}
 \put(23,83){\vector(1,1){0}}
 \put(43,60){\vector(1,0){0}}
 \put(63,43){\vector(1,1){0}}
 \put(63,77){\vector(1,-1){0}}
 \put(17,83){\makebox(0,0)[br]{$a$}}
 \put(63,83){\makebox(0,0)[bl]{$b$}}
 \put(17,37){\makebox(0,0)[tr]{$b$}}
 \put(63,37){\makebox(0,0)[tl]{$a*b$}}
 \put(43,63){\makebox(0,0)[bl]{$ab$}}
 \put(0,105){\makebox(0,0){(C)}}
 \put(40,5){\makebox(0,0){$\langle a,b\rangle-\langle b,a*b\rangle$}}
\end{picture}
\hfill\mbox{}
\caption{Dividing a square into triangles}
\label{dual}
\end{figure}


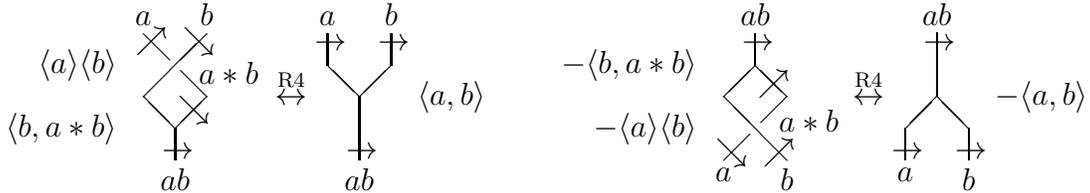
\begin{figure}[htb]
\mbox{}\hfill
\begin{minipage}{92pt}
\begin{picture}(92,68)
 \put(50,34){\line(1,1){24}}
 \put(74,34){\line(-1,1){10}}
 \put(60,48){\line(-1,1){10}}
 \put(62,22){\line(-1,1){12}}
 \put(62,22){\line(1,1){12}}
 \put(62,10){\line(0,1){12}}
 \put(53,55){\makebox(0,0){$\nearrow$}}
 \put(71,55){\makebox(0,0){$\searrow$}}
 \put(69,29){\makebox(0,0){$\searrow$}}
 \put(63,13){\makebox(0,0){$\rightarrow$}}
 \put(50,61){\makebox(0,0)[b]{$a$}}
 \put(74,61){\makebox(0,0)[b]{$b$}}
 \put(71,39){\makebox(0,0)[bl]{$a*b$}}
 \put(62,7){\makebox(0,0)[t]{$ab$}}
 \put(40,46){\makebox(0,0)[r]{$\langle a\rangle\langle b\rangle$}}
 \put(40,22){\makebox(0,0)[r]{$\langle b,a*b\rangle$}}
\end{picture}
\end{minipage}
~$\overset{\text{R4}}{\leftrightarrow}$~
\begin{minipage}{56pt}
\begin{picture}(56,68)
 \put(0,46){\line(0,1){12}}
 \put(24,46){\line(0,1){12}}
 \put(12,34){\line(-1,1){12}}
 \put(12,34){\line(1,1){12}}
 \put(12,10){\line(0,1){24}}
 \put(1,54){\makebox(0,0){$\rightarrow$}}
 \put(25,54){\makebox(0,0){$\rightarrow$}}
 \put(13,13){\makebox(0,0){$\rightarrow$}}
 \put(0,61){\makebox(0,0)[b]{$a$}}
 \put(24,61){\makebox(0,0)[b]{$b$}}
 \put(12,7){\makebox(0,0)[t]{$ab$}}
 \put(34,34){\makebox(0,0)[l]{$\langle a,b\rangle$}}
\end{picture}
\end{minipage}
\hfill
\begin{minipage}{99pt}
\begin{picture}(99,68)
 \put(70,46){\line(0,1){12}}
 \put(82,34){\line(-1,1){12}}
 \put(58,34){\line(1,1){12}}
 \put(58,10){\line(1,1){10}}
 \put(72,24){\line(1,1){10}}
 \put(82,10){\line(-1,1){24}}
 \put(71,54){\makebox(0,0){$\rightarrow$}}
 \put(77,39){\makebox(0,0){$\nearrow$}}
 \put(79,13){\makebox(0,0){$\nearrow$}}
 \put(61,13){\makebox(0,0){$\searrow$}}
 \put(70,61){\makebox(0,0)[b]{$ab$}}
 \put(79,29){\makebox(0,0)[tl]{$a*b$}}
 \put(82,7){\makebox(0,0)[t]{$b$}}
 \put(58,7){\makebox(0,0)[t]{$a$}}
 \put(48,22){\makebox(0,0)[r]{$-\langle a\rangle\langle b\rangle$}}
 \put(48,46){\makebox(0,0)[r]{$-\langle b,a*b\rangle$}}
\end{picture}
\end{minipage}
~$\overset{\text{R4}}{\leftrightarrow}$~
\begin{minipage}{64pt}
\begin{picture}(64,68)
 \put(12,34){\line(0,1){24}}
 \put(24,22){\line(-1,1){12}}
 \put(0,22){\line(1,1){12}}
 \put(24,10){\line(0,1){12}}
 \put(0,10){\line(0,1){12}}
 \put(13,54){\makebox(0,0){$\rightarrow$}}
 \put(25,13){\makebox(0,0){$\rightarrow$}}
 \put(1,13){\makebox(0,0){$\rightarrow$}}
 \put(12,61){\makebox(0,0)[b]{$ab$}}
 \put(24,7){\makebox(0,0)[t]{$b$}}
 \put(0,7){\makebox(0,0)[t]{$a$}}
 \put(34,34){\makebox(0,0)[l]{$-\langle a,b\rangle$}}
\end{picture}
\end{minipage}
\hfill\mbox{}
\caption{Colors for Y-oriented R4 moves}
\label{crossfinger}
\end{figure}

On the other hand, this equation corresponds to Y-oriented R4 moves in Figure~\ref{fig:Rmove} as follows.
In Figure~\ref{crossfinger}, colored diagrams of Y-oriented R4 moves are depicted.
In the left diagram, the left-hand side represents the chain
$\langle a\rangle\langle b\rangle+\langle b,a*b\rangle$
and the right-hand side represents $\langle a,b\rangle$, respectively.
In the right diagram, the left-hand side represents the chain
$-\langle a\rangle\langle b\rangle-\langle b,a*b\rangle$
and the right-hand side represents $-\langle a,b\rangle$, respectively.
Thus the above equality is needed for colored diagrams to define equivalent chains in the quotient complex.

\begin{figure}[htb]
\begin{center}
\includegraphics[height=1.2in]{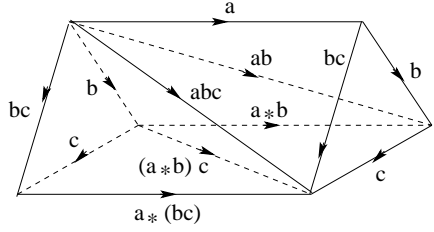}
\end{center}
\caption{Decomposition of a prism into tetrahedra}
\label{prism}
\end{figure}

A geometric interpretation of the last expression of $D_3(X)_Y$ omitting $\langle y\rangle$,
\[ \langle a\rangle\langle b,c\rangle-\langle a,b,c\rangle
+\langle b,a*b,c\rangle-\langle b,c,a*(bc)\rangle \]
is found in Figure~\ref{prism}.
The symbol $\langle a\rangle$ is represented by the horizontal $1$-simplex, $\langle b,c\rangle$ is represented by the right triangular face, and $\langle a\rangle\langle b,c\rangle$ is represented by a prism.
The term $\langle a,b,c\rangle$ corresponds to the right top tetrahedron in the prism.
The expressions of the form
$\langle\langle a\rangle\langle b,c\rangle\rangle_i$
provides a triangulation of a product of simplices.
Each term corresponds to
\begin{align*}
&\langle a,b,c\rangle=\langle\langle a\rangle\langle b,c\rangle\rangle_1,
&&\langle b,a*b,c\rangle=-\langle\langle a\rangle\langle b,c\rangle\rangle_2,
&&\langle b,c,a*(bc)\rangle=\langle\langle a\rangle\langle b,c\rangle\rangle_3.
\end{align*}

Below we use the notation
\begin{align*}
&\partial_{(0)}\langle x_1,\ldots,x_m\rangle
=*x_1\langle x_2,\ldots,x_m\rangle, \\
&\partial_{(i)}\langle x_1,\ldots,x_m\rangle
=(-1)^i\langle x_1,\ldots,x_ix_{i+1},\ldots,x_m\rangle, \\
&\partial_{(m)}\langle x_1,\ldots,x_m\rangle
=(-1)^m\langle x_1,\ldots,x_{m-1}\rangle.
\end{align*}
Then the boundaries of
$\langle\langle a\rangle\langle b,c\rangle\rangle_i$
are computed as follows.
\[ \langle\langle a\rangle\langle b,c\rangle\rangle_i
\stackrel{\partial}{\mapsto}
\partial_{(0)}\langle\langle a\rangle\langle b,c\rangle\rangle_i
+\partial_{(1)}\langle\langle a\rangle\langle b,c\rangle\rangle_i
+\partial_{(2)}\langle\langle a\rangle\langle b,c\rangle\rangle_i
+\partial_{(3)}\langle\langle a\rangle\langle b,c\rangle\rangle_i \]
and the right-hand sides for $i=1,2,3$ are computed as follows.
\[ \begin{array}{rlll}
\langle\langle a\rangle\langle b,c\rangle\rangle_1\stackrel{\partial}{\mapsto}
&*a\langle b,c\rangle & -\langle ab,c\rangle \\
&+\langle a,bc\rangle & -\langle a,b\rangle \\
=&\langle(\partial_{(0)}\langle a\rangle)\langle b,c\rangle\rangle_1
&+\partial_{(1)}\langle\langle a\rangle\langle b,c\rangle\rangle_1 \\
&-\langle \langle a \rangle \partial_{(1)} \langle b,c \rangle \rangle_1
&-\langle \langle a \rangle \partial_{(2)} \langle b,c \rangle \rangle_1,
\\
\langle\langle a\rangle\langle b,c\rangle\rangle_2\stackrel{\partial}{\mapsto}
&-*b\langle a*b,c\rangle & +\langle b(a*b),c\rangle \\
&-\langle b,(a*b)c\rangle & +\langle b,a*b\rangle \\
=&-\langle\langle a\rangle\partial_{(0)}\langle b,c\rangle\rangle_1
&-\partial_{(1)}\langle\langle a\rangle\langle b,c\rangle\rangle_1 \\
&-\partial_{(2)}\langle\langle a\rangle\langle b,c\rangle\rangle_3
&-\langle\langle a\rangle\partial_{(2)}\langle b,c\rangle\rangle_2,
\\
\langle\langle a\rangle\langle b,c\rangle\rangle_3\stackrel{\partial}{\mapsto}
&*b\langle c,a*(bc)\rangle & -\langle bc,a*(bc)\rangle \\
&+\langle b,c(a*(bc))\rangle & -\langle b,c\rangle \\
=&-\langle\langle a\rangle\partial_{(0)}\langle b,c\rangle\rangle_2
&-\langle\langle a\rangle\partial_{(1)}\langle b,c\rangle\rangle_2 \\
&+\partial_{(2)}\langle\langle a\rangle\langle b,c\rangle\rangle_3
&+\langle(\partial_{(1)}\langle a\rangle)\langle b,c\rangle\rangle_1,
\end{array} \]
\begin{sloppypar}
\noindent
where $\langle(\partial_{(i)}\langle a\rangle)\langle b,c\rangle\rangle_1$
is regarded as
$(\partial_{(i)}\langle a\rangle)\langle b,c\rangle$.
The canceling terms of the form
$\partial_{(i)}\langle\langle a\rangle\langle b,c\rangle\rangle_j$
in the above boundaries correspond to internal triangles in Figure~\ref{prism}, that are shared by a pair of tetrahedra.
Other terms are of the form
$\langle \partial_{(i)} \langle a \rangle \langle b,c \rangle \rangle_{j}$ or
$\langle \langle a \rangle \partial_{(i)} \langle b,c \rangle \rangle_{j}$,
and they are outer triangles, that constitute the boundary of the prism.
The expression $\langle\partial_{(i)}\langle a\rangle\langle b,c\rangle\rangle_{j}$ represents the two triangles on the right and the left in Figure~\ref{prism}, since this represents
\end{sloppypar}
\vspace{-15pt} 
\begin{align*}
&(\text{boundary of the interval represented by $\langle a\rangle$}) \\
&\hspace{20mm}\times(\text{the triangle represented by $\langle b,c\rangle$}).
\end{align*}
Thus the outer boundary follows the pattern of Leibniz rule.

In terms of the coloring invariant of graphs, as in the case of the preceding relation for Y-oriented R4 move, this relation corresponds to an equivalence of colored $2$-complexes called foams, that is a higher dimensional analog of the move depicted in Figure~\ref{crossfinger}.
See \cite{CI} for more on colored foams.

\begin{lemma}
\label{lem:decomp}
For $\langle\boldsymbol{a}\rangle=\langle a_1,\ldots,a_s\rangle$ and
$\langle\boldsymbol{b}\rangle=\langle b_1,\ldots,b_t\rangle$
where $a_i,b_j\in G_\lambda$, we have
\[ \partial\langle\langle\boldsymbol{a}\rangle
\langle\boldsymbol{b}\rangle\rangle
=\langle(\partial\langle\boldsymbol{a}\rangle)
\langle\boldsymbol{b}\rangle\rangle
+(-1)^{|\boldsymbol{a}|}\langle\langle\boldsymbol{a}\rangle
(\partial\langle\boldsymbol{b}\rangle)\rangle, \]
where $\langle \langle\cdot\rangle \langle\cdot\rangle \rangle$ is linearly extended.
\end{lemma}

\begin{proof}
By definition, we have
\[ \partial
\langle\langle\boldsymbol{a}\rangle\langle\boldsymbol{b}\rangle\rangle
=\sum_{i=0}^{s+t}\partial_{(i)}
\langle\langle\boldsymbol{a}\rangle\langle\boldsymbol{b}\rangle\rangle
=\sum_{i=0}^{s+t}\sum_{1\leq i_1<\cdots<i_s\leq s+t}\partial_{(i)}
\langle\langle\boldsymbol{a}\rangle\langle\boldsymbol{b}\rangle\rangle_{i_1,\ldots,i_s}. \]
Direct computations show that
\begin{align*}
&\partial_{(0)}
\langle\langle\boldsymbol{a}\rangle\langle\boldsymbol{b}\rangle\rangle_{i_1,\ldots,i_s} \\
&=\begin{cases}
\langle(\partial_{(0)}
\langle\boldsymbol{a}\rangle)\langle\boldsymbol{b}\rangle
\rangle_{i_2-1,\ldots,i_s-1} & (i_1=1), \\
(-1)^s\langle
\langle\boldsymbol{a}\rangle(\partial_{(0)}\langle\boldsymbol{b}\rangle)
\rangle_{i_1-1,\ldots,i_s-1} & (i_1>1),
\end{cases} \\
&\partial_{(i)}\langle
\langle\boldsymbol{a}\rangle\langle\boldsymbol{b}\rangle
\rangle_{i_1,\ldots,i_s} \\
&=\begin{cases}
\langle(\partial_{(k)}
\langle\boldsymbol{a}\rangle)\langle\boldsymbol{b}\rangle
\rangle_{i_1,\ldots,i_k,i_{k+2}-1,\ldots,i_s-1} & (i_k=i<i+1=i_{k+1}), \\
-\partial_{(i)}\langle
\langle\boldsymbol{a}\rangle\langle\boldsymbol{b}\rangle
\rangle_{i_1,\ldots,i_{k-1},i_k+1,i_{k+1},\ldots,i_s} & (i_k=i<i+1<i_{k+1}), \\
-\partial_{(i)}\langle
\langle\boldsymbol{a}\rangle\langle\boldsymbol{b}\rangle
\rangle_{i_1,\ldots,i_k,i_{k+1}-1,i_{k+2},\ldots,i_s} & (i_k<i<i+1=i_{k+1}), \\
(-1)^s\langle
\langle\boldsymbol{a}\rangle(\partial_{(i-k)}\langle\boldsymbol{b}\rangle)
\rangle_{i_1,\ldots,i_k,i_{k+1}-1,\ldots,i_s-1} & (i_k<i<i+1<i_{k+1}),
\end{cases} \\
&\partial_{(s+t)}\langle
\langle\boldsymbol{a}\rangle\langle\boldsymbol{b}\rangle
\rangle_{i_1,\ldots,i_s} \\
&=\begin{cases}
\langle(\partial_{(s)}
\langle\boldsymbol{a}\rangle)\langle\boldsymbol{b}\rangle
\rangle_{i_1,\ldots,i_{s-1}} & (i_s=s+t), \\
(-1)^s\langle
\langle\boldsymbol{a}\rangle(\partial_{(t)}\langle\boldsymbol{b}\rangle)
\rangle_{i_1,\ldots,i_s} & (i_s<s+t).
\end{cases}
\end{align*}

\begin{sloppypar}
The terms of the form
$-\partial_{(i)}\langle
\langle\boldsymbol{a}\rangle\langle\boldsymbol{b}\rangle
\rangle_{i_1,\ldots,i_{k-1},i_k+1,i_{k+1},\ldots,i_s}$
($i_k=i<i+1<i_{k+1}$) and
$-\partial_{(i)}\langle
\langle\boldsymbol{a}\rangle\langle\boldsymbol{b}\rangle
\rangle_{i_1,\ldots,i_k,i_{k+1}-1,i_{k+2},\ldots,i_s}$
($i_k<i<i+1=i_{k+1}$) cancel in pairs.
The other terms are organized as
\begin{align*}
&\sum_{1\leq i_1<\cdots<i_{s-1}\leq s+t-1}\sum_{i=0}^s
\langle
(\partial_{(i)}\langle\boldsymbol{a}\rangle)\langle\boldsymbol{b}\rangle
\rangle_{i_1,\ldots,i_{s-1}}
+\sum_{1\leq i_1<\cdots<i_s\leq s+t-1}\sum_{i=0}^{t}
(-1)^s \langle
\langle\boldsymbol{a}\rangle (\partial_{(i)}\langle\boldsymbol{b}\rangle)
\rangle_{i_1,\ldots,i_s} \\
&=
\sum_{1\leq i_1<\cdots<i_{s-1}\leq s+t-1}\langle
(\partial\langle\boldsymbol{a}\rangle) \langle\boldsymbol{b}\rangle
\rangle_{i_1,\ldots,i_{s-1}}
+\sum_{1\leq i_1<\cdots<i_s\leq s+t-1}(-1)^s\langle
\langle\boldsymbol{a}\rangle (\partial\langle\boldsymbol{b}\rangle)
\rangle_{i_1,\ldots,i_s} \\
&=
\langle(\partial\langle\boldsymbol{a}\rangle)
\langle\boldsymbol{b}\rangle\rangle
+(-1)^s\langle\langle\boldsymbol{a}\rangle
(\partial\langle\boldsymbol{b}\rangle)\rangle,  
\end{align*}
where $\langle\,\cdot\,\rangle_{i_1,\ldots,i_s}$ is linearly extended.
\end{sloppypar}
\end{proof}

Since the Leibniz rule holds
(the preceding Lemma~\ref{lem:decomp}), 
we have the following.

\begin{lemma} \label{lem:subcomplexM}
$D_*(X)_Y=(D_n(X)_Y,\partial_n)$ is a subcomplex of $P_*(X)_Y$.
\end{lemma}

\section{Chain map for simplicial decomposition}
\label{sec:G}

In this section we examine relations between group and MCQ homology theories.

\subsection{Simplicial decomposition (general case)}

We observe an associativity of the notation 
$\langle\langle\boldsymbol{a}\rangle \langle\boldsymbol{b}\rangle\rangle$ 
defined in Section~\ref{sec:homology}, and extend the notation to multi-tuples.  
For an expression of the form
$\langle\boldsymbol{a}\rangle \langle\boldsymbol{b}\rangle \langle\bm{c}\rangle$ 
in a chain in $P_*(X)_Y$, where
$\bm{a},\bm{b},\bm{c} \in \bigcup_{m \in \N} G_\lambda^m$, 
it is easy to see that we have the following. 

\begin{lemma}\label{lem:comm}
$
\langle
\langle \langle\bm{a}\rangle \langle\bm{b}\rangle \rangle 
\langle\bm{c}\rangle 
\rangle
=
\langle
\langle\bm{a}\rangle 
\langle \langle\bm{b}\rangle \langle\bm{c}\rangle \rangle
\rangle
$. 
\end{lemma}
By Lemma~\ref{lem:comm}, we can define 
$\langle 
\langle\bm{a}\rangle \langle\bm{b}\rangle \langle\bm{c}\rangle
\rangle$ 
by 
$\langle
\langle \langle\bm{a}\rangle \langle\bm{b}\rangle \rangle 
\langle\bm{c}\rangle 
\rangle
=
\langle
\langle\bm{a}\rangle 
\langle \langle\bm{b}\rangle \langle\bm{c}\rangle \rangle
\rangle$. 
Moreover, for an expression of the form
$\langle\bm{a}_1\rangle \cdots \langle\bm{a}_k\rangle$ 
in a chain in $P_*(X)_Y$, where 
$\bm{a}_1, \ldots, \bm{a}_k \in \bigcup_{m \in \N} G_\lambda^m$, 
we can define 
$\langle
\langle\bm{a}_1\rangle \cdots \langle\bm{a}_k\rangle
\rangle$ 
inductively. 
Then, by Lemma~\ref{lem:decomp}, this notation is compatible with the boundary homomorphism $\partial$ 
in the following sense. 
\begin{lemma}\label{lem:bigcomm} %
$ 
\partial \langle
\langle\bm{a}_1\rangle \cdots \langle\bm{a}_k\rangle
\rangle
=
\langle \partial(
\langle\bm{a}_1\rangle \cdots \langle\bm{a}_k\rangle
)\rangle. 
$ 
\end{lemma}
We give a direct formula (instead of induction) for 
the notation 
$\langle
\langle\bm{a}_1\rangle \cdots \langle\bm{a}_k\rangle
\rangle$ 
later in Subsection~\ref{subsec:decomp-direct}.

\subsection{Chain map (from MCQ to Group)}
\label{subsec:G}

Let $X=\coprod_{\lambda\in\Lambda}G_\lambda$ be a multiple conjugation quandle, 
and let $Y$ be an $X$-set.
Let $P_n^G(X)_Y$ be the subgroup of $P_n(X)_Y$ generated by the elements of the form 
$\langle y \rangle \langle \bm{a} \rangle$. 
Let $D_n^G(X)_Y$ and $D_n^{G,\uparrow}(X)_Y$ be respectively 
$P_n^G(X)_Y \cap D_n(X)_Y$ and $P_n^G(X)_Y \cap D_n^{\uparrow}(X)_Y$, 
both of which are the subgroups of $P_n^G(X)_Y$. 
Note that $D_n^G(X)_Y=P_n^G(X)_Y \cap D_n(X)_Y$ is the trivial group.  
We put 
\begin{align*}
&C_n^G(X)_Y := P_n^G(X)_Y/D_n^G(X)_Y =P_n^G(X)_Y, \\ 
&C_n^{G,\uparrow}(X)_Y := P_n^{G}(X)_Y/(D_n^{G}(X)_Y + D_n^{G,\uparrow}(X)_Y)  
= P_n^{G}(X)_Y/D_n^{G,\uparrow}(X)_Y .
\end{align*}
Then $C_*^G(X)_Y=(C_n^G(X)_Y,\partial_n)$ and
$C_*^{G,\uparrow}(X)_Y=(C_n^{G,\uparrow}(X)_Y,\partial_n)$ are chain complexes.
If $X$ is a group (regarded as $X=\coprod_{\lambda\in\Lambda}G_\lambda$ with 
$\Lambda$ a singleton) and $Y$ is a singleton, 
$C_*^G(X)_Y$ is essentially the same as the chain complex of usual group homology.
For an abelian group $A$, we define the cochain complexes
\begin{align*}
&C^*_G(X;A)_Y=\operatorname{Hom}(C_*^G(X)_Y,A)
\ {\rm and} \ 
C^*_{G,\uparrow}(X;A)_Y=\operatorname{Hom}(C_*^{G,\uparrow}(X)_Y,A).
\end{align*}

When $X$ is a multiple conjugation quandle consisting of a single group, 
define  homomorphisms 
$\Delta: P_*(X)_Y \to P_*^G(X)_Y$ by 
$$\Delta (
\langle\bm{a}_{1}\rangle \cdots \langle\bm{a}_{m}\rangle
):=
\langle \langle\bm{a}_{1}\rangle \cdots \langle\bm{a}_{m}\rangle \rangle . 
$$
Then by Lemma~\ref{lem:bigcomm} and from definitions, we have the following.
\begin{proposition}\label{prop:MCQ2G}
The homomorphisms $\Delta: P_*(X)_Y \to P_*^G(X)_Y$ give rise to a chain homomorphism.
Furthermore, $\Delta$  induces the chain homorphisms $\Delta: C_*(X)_Y \to C_*^G(X)_Y$  and  $\Delta: C_*^{\uparrow}(X)_Y \to C_*^{G,\uparrow}(X)_Y$. 
\end{proposition}
We note that the chain homomorphisms $\Delta$ are defined only for a MCQ 
consisting of a single group. 
In this case, we also have the cochain homomorphisms 
$\Delta: C^*_{G}(X;A)_Y \to C^*(X;A)_Y$ 
and 
$\Delta: C^*_{G,\uparrow}(X;A)_Y \to C^*_{\uparrow}(X;A)_Y$ 
for an abelian group $A$. 
Hence, for a given cocycle of group homology theory, 
we can obtain that of our theory through $\Delta$. 
This approach will be discussed in Section~\ref{sec:attempt}.

\begin{remark} 
We point out here that for a group $X=\Z_3$ and a trivial $X$-set $Y$, 
there is a group $2$-cocycle $\eta$
that satisfies the conditions 
 in $C^2_{G,\uparrow}(X)_Y$ 
 (coming from $D_n^{G,\uparrow}(X)_Y$), 
$$\eta \langle  a,b\rangle + \eta \langle  a^{-1},ab\rangle =0
\quad  {\rm and} \quad
\eta \langle  a,b\rangle + \eta \langle   ab,b^{-1}\rangle =0 .
$$
Specifically, let $\eta: \Z_3 \times \Z_3 \rightarrow \Z_3$ denote the function 
that  has  values $\eta(1,1)=1$, $\eta(2,2)=2$ and $\eta(g, h)=0$ otherwise. 
It is a direct calculation that the condition above is satisfied. 
Furthermore, to see that $\eta$ is a cocycle, consider the generating cocycle 
over $G=\Z_p$ where   $p$ is a prime that is defined by 
$$\eta_0 (x, y) = (1/p) (\overline{x} + \overline{y} - \overline{x+y} ) \pmod{p}, $$
where $\overline{x}$ is an integer $0 \leq \overline{x} < p $ such that $\overline{x}= x \pmod{p}$.
It is known that $\eta_0$ is a generating $2$-cocycle for 
$H^2_G(\Z_p; \Z_p)$
for prime $p$.
For $p=3$,
let $\zeta $ be a $1$-chain defined by $\zeta(0)=0$ and $\zeta(1)+\zeta(2)=2$.
Then  one can easily compute that  $\eta=\eta_0 + \delta \zeta.$ 
Hence there is a $2$-cocycle $\eta \in C^2_{G,\uparrow}(X)_Y$  of our theory
that is cohomologous to the standard group $2$-cocyle $\eta_0$.
\end{remark}

\subsection{Simplicial decomposition (direct formula)}
\label{subsec:decomp-direct}

We give a direct formula (instead of induction) for 
the notation 
$\langle
\langle\bm{a}_1\rangle \cdots \langle\bm{a}_k\rangle
\rangle$.  
To the term 
$
\langle\langle\boldsymbol{a}\rangle  \langle\boldsymbol{b}\rangle 
\rangle_{i_1,\ldots,i_s}
$, 
we associate a vector $\boldsymbol{v} = (v_1, \ldots, v_n) \in \{ 1, 2\}^{n} $
by defining $v_i=1 $ if $i=i_j$ for some $j$, and otherwise $v_i=2$, where $n=s+t$. 
In the term
\[ c_i=\begin{cases}
a_k*(b_1\cdots b_{i-k}) & \text{if $i=i_k$,} \\
b_{i-k} & \text{if $i_k<i<i_{k+1}$,}
\end{cases} \]
the first entry with $a_k$  in it corresponds to $v_i=1$ and the second with $b_{i-k} $ to $v_i=2$. 
We regard that the term $a_k$  came from the first part $\langle\boldsymbol{a}\rangle $ in 
$
\langle\langle\boldsymbol{a}\rangle  \langle\boldsymbol{b}\rangle 
\rangle_{i_1,\ldots,i_s}
$ so that $v_i=1$ is assigned, 
and the term $b_{i-k} $ belongs to the second part  $\langle\boldsymbol{b}\rangle$
receiving $v_i=2$. 

\begin{example}
{\rm 
For the term $\langle a \rangle \langle b, c  \rangle$ discussed for Figure~\ref{prism}, 
the terms 
$  \langle a,b,c \rangle  $, $ -  \langle b,a*b,c \rangle  $, and $  \langle b,c,a*(bc)  \rangle $
correspond to  the vectors $(1,2,2)$, $(2,1,2)$, and $(2,2,1)$, respectively. 
Note that  $(2,1,2)$ is obtained from $(1,2,2)$ by a transposition of the first two entries, and this fact reflects in Figure~\ref{prism} that the tetrahedra represented by these vectors share a triangular internal  face.
We indicate by an edge between two vectors when one is obtained from the other by a transposition
of consecutive entries. 
In this case we write:
\[ \begin{array}{c@{~}c@{~}c@{~}c@{~}c}
(1,2,2) & - & (2,1,2) & - & (2,2,1).
\end{array} \]

For $\langle a,b\rangle \langle c,d \rangle$, 
the terms 
$
\langle\langle\boldsymbol{a}\rangle  \langle\boldsymbol{b}\rangle 
\rangle_{i_1,\ldots,i_s}
$
are listed as 
$\langle  a,b,c,d  \rangle $, 
$- \langle  a,c,b*c,d  \rangle $,
$\langle   c, a*c, b*c, d\rangle $,
$\langle  a, c, d, b*(cd)   \rangle $,
$- \langle  c, a*c, d, b*(cd)  \rangle $,
$\langle  c,  d, a*(cd) , b*(cd)  \rangle $,
and these correspond to vectors
$$( 1,1,2,2), \; (1,2,1,2), \; (2,1,1,2), \; (1,2,2,1),\; (2,1,2,1),\; (2,2,1,1), $$
respectively. They are connected by edges as 
\[ \begin{array}{c@{\,}c@{\,}c@{\,}c@{\,}c@{\,}c@{\,}c@{\,}c@{\,}c@{\,}c}
& & & & (2,1,1,2) & & & & \\
& & & \diagup & & \diagdown & & & \\
(1,1,2,2) & - & (1,2,1,2) & & & &(2,1,2,1) & - & (2,2,1,1) \\
& & & \diagdown & & \diagup & & & \\
& & & &  (1,2,2,1) & & & &
\end{array} \]
indicating which simplices share internal faces.
Note that from a vector 
$\boldsymbol{v} = (v_1, \ldots, v_n) \in \{ 1, 2\}^{n} $
 the subscripts $i_1, \ldots , i_s$ in $\langle\langle\boldsymbol{a}\rangle  \langle\boldsymbol{b}\rangle 
\rangle_{i_1,\ldots,i_s}$ are recovered by the condition $v_{i_j}=1$.

}
\end{example}

\begin{sloppypar}
For an expression of the form
$\langle\bm{a}_1\rangle \cdots \langle\bm{a}_k\rangle$ 
in a chain in $P_*(X)_Y$, where 
$\bm{a}_1, \ldots, \bm{a}_k \in \bigcup_{m \in \N} G_\lambda^m$, 
we put $n=|\bm{a}_1|+\cdots+|\bm{a}_k|$ and consider
vectors  $\boldsymbol{v}=(v_1,\ldots,v_n)\in\{1,\ldots,k\}^n$, and denote 
by $\#_j^i\boldsymbol{v}$ the number of $j$'s in $v_1,\ldots,v_i$.
Then for a given  $\boldsymbol{v}$ define $i(j, 1) < \cdots < i(j, n_j)$ by the condition that 
$v_{i(j, 1)} = \cdots =  v_{i(j, n_j)} = j$. 
\end{sloppypar}

With these notations in hand, we temporarily define 
$\langle \langle\bm{a}_1\rangle \cdots \langle\bm{a}_k\rangle \rangle'$ by
\[ \sum_{\begin{subarray}{c}
\boldsymbol{v}\in\{1,\ldots,k\}^n \\
\#_j^n \boldsymbol{v}=n_j ~(j=1,\ldots,k)
\end{subarray}}\hspace{-3mm}
(-1)^{\sum_{j=1}^{k-1}\sum_{t=1}^{n_j}(i(j,t)-t-\sum_{s=1}^{j-1}n_s)}
\langle c_1,\ldots,c_n\rangle \]
for $\langle\boldsymbol{a}_1\rangle\cdots\langle\boldsymbol{a}_k\rangle
=\langle a_{1,1},\ldots,a_{1,n_1}\rangle
\cdots\langle a_{k,1},\ldots,a_{k,n_k}\rangle$, 
where 
\[ c_i=a_{v_i,\#_{v_i}^i\boldsymbol{v}}
*\prod_{s=v_i+1}^k\prod_{t=1}^{\#_s^i\boldsymbol{v}}a_{s,t}. \]

Then we have 
$\langle \langle\bm{a}_1\rangle \cdots \langle\bm{a}_k\rangle \rangle'
=
\langle \langle\bm{a}_1\rangle \cdots \langle\bm{a}_k\rangle \rangle. $
This is seen from the fact that simplices of both sides 
are in one-to-one correspondence with vectors $\boldsymbol{v}=(v_1,\ldots,v_n)\in\{1,\ldots,k\}^n$,
and the signs correspond to the number of transpositions, modulo 2, 
of a given vector  $\boldsymbol{v}$ from the vector 
$(1, \ldots, 1, 2, \ldots, 2, \ldots, k, \ldots, k )$.

\begin{figure}[htb]
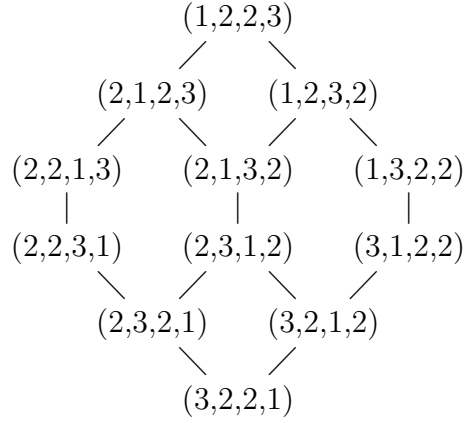

\[ \begin{array}{ccccccccc}
&&&& \makebox(0,5){(1,2,2,3)} &&&& \\
&&& \diagup && \diagdown &&& \\
&& \makebox(0,5){(2,1,2,3)} &&&& \makebox(0,5){(1,2,3,2)} && \\
& \diagup && \diagdown && \diagup && \diagdown & \\
\makebox(0,5){(2,2,1,3)} &&&& \makebox(0,5){(2,1,3,2)} &&&& \makebox(0,5){(1,3,2,2)} \\
| &&&& | &&&& | \\
\makebox(0,5){(2,2,3,1)} &&&& \makebox(0,5){(2,3,1,2)} &&&& \makebox(0,5){(3,1,2,2)} \\
& \diagdown && \diagup && \diagdown && \diagup & \\
&& \makebox(0,5){(2,3,2,1)} &&&& \makebox(0,5){(3,2,1,2)} && \\
&&& \diagdown && \diagup &&& \\
&&&& \makebox(0,5){(3,2,2,1)} &&&&
\end{array} \]
\caption{Boundaries of $\langle a\rangle\langle b,c\rangle\langle d\rangle$}
\label{boundary}
\end{figure}

\begin{example}
{\rm

The terms  of 
$\langle 
\langle a \rangle  \langle b, c \rangle  \langle d \rangle
\rangle$
consist of
$$
\begin{array}{llll}
 \langle a, b, c, d  \rangle, &  \langle b, a*b , c, d  \rangle, & \langle a, b, d, c*d  \rangle, \\
  \langle b,c, a*(bc), d \rangle, &   \langle b, a*b, d, c*d  \rangle, & \langle a,d,b*d,c*d  \rangle, \\
 \langle b,c,d, a*(bcd)  \rangle, &  \langle  b,d,a*(bd), c*d  \rangle, &  \langle d, a*d, b*d, c*d  \rangle, \\
  \langle b,d, c*d, a*(bcd)  \rangle, 
 & \langle d, b*d, a*(bd), c*d  \rangle, & \langle d, b*d, c*d, a*(bcd)  \rangle, 
\end{array}
$$
that, respectively, correspond to the vectors 
\[
\begin{array}{llll}
(1,2,2,3), & (2,1,2,3), & (1,2,3,2), \\
(2,2,1,3), & (2,1,3,2), & (1,3,2,2), \\
(2,2,3,1), & (2,3,1,2), & (3,1,2,2), \\
(2,3,2,1), & (3,2,1,2), &(3,2,2,1).
\end{array}
\]
The graph representing shared faces is depicted in Figure~\ref{boundary}.

}
\end{example}

\section{Relationship between MCQ and IIJO}
\label{sec:IIJO}

Let $X=\coprod_{\lambda\in\Lambda}G_\lambda$ be a multiple conjugation quandle, and let $Y$ be an $X$-set.
Let $P_n^{\rm IIJO}(X)_Y$ be the subgroups of $P_n(X)_Y$ generated by the elements of the form
$\langle y\rangle\langle a_1\rangle\cdots\langle a_n\rangle$.
Then $P^{\rm IIJO}_*(X)_Y=(P^{\rm IIJO}_n(X)_Y,\partial_n)$ is a subcomplex of $P_*(X)_Y$.
Let $D_n^{\rm IIJO}(X)_Y$ be the subgroup of $P_n^{\rm IIJO}(X)_Y$ generated by the elements of the forms
\begin{align*}
&\langle y\rangle\langle a_1\rangle\cdots\langle b_1\rangle\langle b_2\rangle\cdots\langle a_n\rangle,
&&\langle y\rangle\langle a_1\rangle\cdots\partial\langle b_1,b_2\rangle\cdots\langle a_n\rangle
\end{align*}
for $a_1,\ldots,a_n \in X$, $b_1,b_2\in G_\lambda$.
We note that the former elements relate with the invariance under the R1 and R4 move, and that the latter elements relate with the invariance under the R5 move and reversing orientation.

\begin{lemma}
$D^{\rm IIJO}_*(X)_Y=(D^{\rm IIJO}_n(X)_Y,\partial_n)$ is a subcomplex of $P^{\rm IIJO}_*(X)_Y$.
\end{lemma}

\begin{proof}
This follows from
\begin{align*}
&\partial(\langle b_1\rangle\langle b_2\rangle)
=\partial\langle b_1,b_2\rangle-\partial\langle b_2,b_1*b_2\rangle,
&&\partial(\partial\langle b_1,b_2\rangle)=0
\end{align*}
for $b_1,b_2\in G_\lambda$.
\end{proof}

We put
\[ C_n^{\rm IIJO}(X)_Y=P_n^{\rm IIJO}(X)_Y/D_n^{\rm IIJO}(X)_Y. \]
Then $C_*^{\rm IIJO}(X)_Y=(C_n^{\rm IIJO}(X)_Y,\partial_n)$ is a chain complex.
If $X$ is obtained from a $G$-family of quandles as in Example~\ref{ex:mcq from G-family},
$C_*^{\rm IIJO}(X)_Y$ is the chain complex defined in~\cite{IshiiIwakiriJangOshiro13}.
For an abelian group $A$, we define the cochain complexes
\[ C^*_{\rm IIJO}(X;A)_Y=\operatorname{Hom}(C_*^{\rm IIJO}(X)_Y,A). \]

We note that a natural projection 
$\mathrm{pr}_*:P_*(X)_Y\to P_*^{\rm IIJO}(X)_Y$
does not induce a chain homomorphism
$\mathrm{pr}_*:C_*(X)_Y\to C_*^{\rm IIJO}(X)_Y$,
since IIJO homology theory is  invariant under the invariance for reversing orientations. 
See Table~\ref{table:IIJO-CIST}. 
It is seen, however, that this map induces the chain homomorphism
$\mathrm{pr}_*:C_*^\uparrow(X)_Y\to C_*^{\rm IIJO}(X)_Y$
and the cochain homomorphism
$\mathrm{pr}^*:C^*_{\rm IIJO}(X;A)_Y\to C^*_\uparrow(X;A)_Y$
for an abelian group $A$.
Hence, for a given cocycle of IIJO homology thoery (with some modification for a multiple conjugation quandle as above), we can obtain that of our theory through $\mathrm{pr}^*$.
This implies that our invariant is a generalization of the IIJO quandle cocycle invariant.

\begin{table}[htb]
\begin{center}\begin{tabular}{|c||c|c|c|c|}
\hline
IIJO
& $2$-boundary & degenerate $D^{\rm IIJO}_2(X)_Y$ & cancelled by sign & zero by definition \\
\hline\hline
moves
& R3 & R4($\rightsquigarrow$ R1), R5($\rightsquigarrow$ ori.)& R2 & R6 \\
\hline
\end{tabular}\end{center}

\begin{center}\begin{tabular}{|c||c|c|c|c|}
\hline
MCQ 
& $2$-boundary & degenerate $D_2(X)_Y$ & degenerate $D_2^{\uparrow}(X)_Y$ & cancelled by sign \\
\hline\hline
moves
& R3, R5, R6
& R4($\rightsquigarrow$ R1) & orientation 
& R2 \\
\hline
\end{tabular}\end{center}
\caption{Comparison between IIJO theory and MCQ theory}
\label{table:IIJO-CIST}
\end{table}

\section{Towards finding $2$-cocycles}\label{sec:attempt}

We discuss approaches to finding 
$2$-cocycles that are not induced from the IIJO (co)homology theory.
Let $G$ be a group, $M$ a right $G$-module, and $A$ an abelian group.
The module $M$ and the set
$X=M\times G$ ($=\coprod_{x\in M}\{x\}\times G$)
can be considered as a $G$-family of quandles and a multiple conjugation quandle as in Example~\ref{ex:mcq from G-family}, respectively.

We take an $X$-set $Y$ as a singleton $\{y_0\}$ and suppress the notation $\langle y_0 \rangle$.
For a $2$-cocycle $\psi \in P^2(X;A)_Y$, we use the notation $\phi((x,g),(y,h))$ for $\psi(\langle(x,g)\rangle\langle(y,h)\rangle)$ and $\eta_x(g,h)$ for $\psi(\langle(x,g),(x,h)\rangle)$.
Then the $2$-cocycle conditions are written as
\begin{eqnarray*}
(1) & & \eta_x(g,h)+\eta_x(gh,k)=\eta_x(h,k)+\eta_x(g,hk),
\\
(2) & & \phi((x,g),(y,k))+\phi((x,h),(y,k))-\phi((x,gh),(y,k)) \\
& & \hspace{2in}=\eta_x(g,h)-\eta_{x*^ky}(g*k,h*k),
\\
(3) & & \phi((x,g),(y,h))+\phi((x*^hy,g*h),(y,k))=\phi((x,g),(y,hk)),
\\
(4) & & \phi((x,g),(y,h))+\phi((x*^hy,g*h),(z,k)) \\
& & \hspace{.5in}=\phi((x,g),(z,k))+\phi((x*^kz,g*k),(y*^kz,h*k)),
\end{eqnarray*}
where $x,y,z\in M$ and $g,h,k\in G$. 
Furthermore, for a $2$-cochain $\psi \in P^2(X;A)_Y$,  
the condition that $\psi$ is a $2$-cochain in $C^2(X;A)_Y$ are written as 
\begin{eqnarray*}
(5) & & \phi((x,g),(x,h)) = \eta_x(g,h)-\eta_{x}(h,g*h), 
\end{eqnarray*}
where $x\in M$ and $g,h\in G$. 

Towards constructing MCQ $2$-cocycles that are not from the IIJO homology, 
first we note that if $\phi$ above is an IIJO  $2$-cocycle, then $\phi$ satisfies the conditions
(3), (4), and the condition that the LHS of (2) vanishes. 
By considering $\psi'=\psi - \phi$, we obtain an MCQ $2$-cocycle $\psi'$ that consists only of terms of $\eta_x$ for  $x \in M$.
Thus we first consider such a case in Example~\ref{ex:gp} below. 
In this case, we can take an approach described in Section~\ref{sec:G} for finding MCQ cocycles from group cocycles.

\begin{example}\label{ex:gp}
For a $2$-cochain $\psi \in P^2(X;A)_Y$ with the assumption 
\begin{eqnarray*}
(0) & & 
 \psi(\langle(x,g)\rangle\langle(y,h)\rangle)\  ( =\phi((x,g), (y,h))) \ = 
 0 , 
\end{eqnarray*}
we discuss what conditions are needed for the $2$-cochain $\psi$ being 
a $2$-cocycle in $P^2(X;A)_Y$.  
When we use the notation
$\eta_x(g,h)$ for $\psi(\langle(x,g),(x,h)\rangle)$, 
the $2$-cocycle conditions are written as
\begin{eqnarray*}
(1) & & \eta_x(g,h)+\eta_x(gh,k)=\eta_x(h,k)+\eta_x(g,hk), \\
(2') & & \eta_x(g,h)-\eta_{x*^k y}(g*k,h*k)=0, 
\end{eqnarray*}
where $x,y\in M$ and $g,h,k\in G$. 
We note that the condition (0) implies (3) and (4). 
Furthermore, for a $2$-cochain $\psi \in P^2(X;A)_Y$ with the assumption (0), 
the condition that $\psi$ is a $2$-cochain in $C^2(X;A)_Y$ are written as 
\begin{eqnarray*}
(5') & & \eta_x(g,h)-\eta_{x}(h,g*h)=0, 
\end{eqnarray*}
where $x\in M$ and $g,h\in G$. 
Hence if $\psi$ satisfies (0), (1), $(2')$ and $(5')$, 
then $\psi$ is a $2$-cocycle in $C^2(X;A)_Y$ and 
defines an invariant for handlebody-knots.

If $y=x$, then $(2')$ implies $\eta_{x}(g*k,h*k)=\eta_x(g,h)$, 
called the \textit{right invariance} of $\eta_x$.
If $x=0$, then $(2')$ with right invariance implies $\eta_{y \cdot (1-k)} \equiv \eta_{0},$ 
which is another necessary condition for the condition $(2')$. 
Hence if any element in $M$ can be represented by the form $y \cdot (1-k)$ for 
some $y \in M$ and $k \in G$, then we have $\eta_x \equiv \eta_0$ for any $x \in M$.  
In this case, we can check that the $2$-cocycle $\psi$ in $C^2(X;A)_Y$ 
comes from the dual of the composition of the chain homomorphisms 
$$C_*(X)_Y \stackrel{\text{pr}_2}{\to} C_*(G)_Y \stackrel{\Delta}{\to} C_*^G(G)_Y , $$
where a chain homomorphism $\text{pr}_2$ is induced from 
a natural projection into $2$nd factor and 
the chain homomorphism $\Delta$ was defined in Subsection~\ref{subsec:G}. 
In this case, $\psi$ assigned at a crossing is decomposed into a pair of weights 
$\eta$ corresponding to trivalent vertices as depicted in Figure~\ref{dual}
 $(B)$ and  $(C)$. 
 Hence the resulting invariant is equivalent to the invariant of  
 the trivalent graph obtained by replacing all crossing with vertices, that is embedded 
in the $2$-sphere without crossing. 
Such an embedded graph is equivalent to a circle with small bubbles, 
and has trivial invariant value
($W(D;C)=0$ for any coloring $C$). 
Thus, in this case,  $\psi$ defines a trivial invariant for handlebody-knots
by the group $2$-cocycle $\eta_0$, 
whose cohomology class may not be zero in $H^2_G(G;A)_Y$.  %

If the condition that any element in $M$ can be represented by the form $y \cdot (1-k)$ 
for some $y \in M$ and $k \in G$ is not satisfied, then $\psi$ 
satisfying (0), (1), $(2')$ and $(5')$ may give rise to a non-trivial invariant 
for handlebody-links. 
\end{example}

\begin{example}
In contrast to Example~\ref{ex:gp}, next we consider the case when $\phi$ is not an IIJO $2$-cocycle, 
so that the LHS of (2) does not vanish for $\phi$.


For any $G$-invariant $A$-bilinear map $f:M^2 \to A$, Nosaka claimed in \cite[Theorem 5.2]{Nosaka13} that the map $\phi_f:X^2\to A$ defined by 
\[ \phi_f((x,g),(y,h)):=f(x-y,y\cdot(1-h^{-1})) \]
satisfies the conditions (3) and (4) above.
For the $G$-invariant $A$-bilinear map $f$, if we can find maps $\eta_x$ such that the conditions (1) and (2) are also satisfied, then we obtain a $2$-cocycle, which may be new.
We remark here that $\phi_f$ itself can be modified in \cite[Corollary 4.7]{Nosaka13} 
(by using an additive homomomorphism form $G$ to some commutative ring) 
so that the conditions (1) and (2) are also satisfied 
under the assumption $\eta_x\equiv0$ for any $x\in M$.

The condition (1) merely says that $\eta_x$ is a usual group $2$-cocycle for any $x \in M$.
The condition (2) is equivalent to
\[ \textrm{($2''$)} \quad \quad 
f(x-y,y\cdot(1-k^{-1}))=\eta_x(g,h)-\eta_{x*^ky}(g*k,h*k) \]
from the definition of $f$. 
If $y=x$, 
then ($2''$) implies that $\eta_x$ is right invariant in the sense that 
$\eta_x(g*k,h*k)=\eta_x(g,h)$ as above. 
If $y=0$, then ($2''$) with the right invariance implies $\eta_{x\cdot k}\equiv\eta_x$, 
called the 
\textit{orbit dependence} of $\eta_x$.
Thus we obtain 
these two necessary conditions for the condition ($2''$).

We examine  the following specific examples. 
For a prime number $p$, let $G=\SL(2,\mathbb{Z}_p)$ that acts on $M=(\mathbb{Z}_p)^2$ from the right.
For $A=\Z_p$, the map $f:M^2 \to A$ defined by
$f(x,y):=\det\begin{pmatrix}x \\ y\end{pmatrix}$ is a $G$-invariant $A$-bilinear map, where $x,y\in M$ are row vectors on which $G$ acts on the right, and $\det$ denotes the determinant.
This setting is motivated from \cite[Proposition 4.5]{Nosaka13}.

\begin{itemize}
\item
First,  
we consider the case where $p=2$.
Define $m:M\to A$ by
\[ m(x):=\begin{cases}
0 & \text{if $x=0$} \\
1 & \text{if $x \neq 0$}
\end{cases}. \]
Then we can check that
\[ \phi_f((x,g),(y,h))=-m(x)+m(x*^hy) \]
for any $x,y \in M$ and $g,h\in G$.
Take $\eta_x(g,h)$ to be $m(x)$ for any $x\in M$ and $g,h\in G$.
Then we can show that the $2$-cochain $\psi$, defined by $\phi_f$ and $\eta_x$, 
is a $2$-coboundary as follows.
Define a $1$-cochain $\tilde{m}\in P^1(X;A)$ by $\tilde{m}(\langle(x,g)\rangle):=m(x)$.
Then $2$-coboundary $\delta \tilde{m} \in P^2(X;A)$ are written as
\begin{align*}
&(\delta\tilde{m})(\langle(x,g)\rangle\langle(y,h)\rangle)=-m(x)+m(x*^hy), \\
&(\delta\tilde{m})(\langle(x,g),(x,h)\rangle)=m(x),
\end{align*}
where $x,y\in M$ and $g,h \in G$.
This implies that $\psi=\delta\tilde{m}$.

\item
Second, 
we consider the case where $p>2$.
If $x=(0,0)$ and 
$k=\begin{pmatrix}-1 & 0 \\ 0 & -1\end{pmatrix}$,
the condition ($2''$) implies $\eta_{2y}(g,h)=\eta_0(g,h)$ for any $y\in M$ and $g,h\in G$.
Since $p$ is odd, we have that $\eta_x\equiv\eta_0$ for any $x\in M$.
If we substitute $y=(1,0)$ and
$k=\begin{pmatrix}1 & -1 \\ 0 & 1\end{pmatrix}$ for $(2'')$, then LHS is $1$ and RHS is $0$, which turns out to be a contradiction.
Hence there is no choice of $\eta_x$ such that the condition (1) and ($2''$) are satisfied.
\end{itemize}
\end{example}

Although our attempts have not resulted in new non-trivial $2$-cocycles, it appears useful to record 
our approaches and facts we have found, for future endeavors towards constructing 
new cocycles using these approaches. 
Further studies are desirable on this homology theory, as it unifies group and quandle 
homology theories for a structure of multiple conjugation quandle, that have ample interesting examples
and applications to handlebody-links.

\section*{Acknowledgements}

The first author was partially supported by Simons Foundation.
The second author was partially supported by JSPS KAKENHI Grant Number 24740037.
The third author was partially supported by  (U.S.) NIH R01GM109459-01.
The last author was partially supported by JSPS KAKENHI Grant Number 26400082.
The authors are grateful to Daniel Moskovich 
 for valuable comments on exposition.

\end{document}